\documentclass[11pt, leqno]{article}

\usepackage{mathrsfs,amssymb,amsmath,amsthm,amsfonts, color, mathscinet}
\usepackage[dvips]{graphicx}

\makeatletter
	
	\@addtoreset{equation}{section}
\makeatother

\begin{document}

\renewcommand{\theenumi}{\rm (\roman{enumi})}
\renewcommand{\labelenumi}{\rm \theenumi}

\newtheorem{thm}{Theorem}[section]
\newtheorem{defi}[thm]{Definition}
\newtheorem{lem}[thm]{Lemma}
\newtheorem{prop}[thm]{Proposition}
\newtheorem{cor}[thm]{Corollary}
\newtheorem{exam}[thm]{Example}
\newtheorem{conj}[thm]{Conjecture}
\newtheorem{rem}[thm]{Remark}
\allowdisplaybreaks

\title{Remarks on Stochastic Systems I:\\ Markov properties, local and global uniqueness, and limits of stochastic equations}

\author{Seiichiro Kusuoka
\vspace{5mm}\\
\small $^*$Department of Mathematics, Graduate School of Science, Kyoto University,\\
\small Kitashirakawa-Oiwakecho, Sakyo-ku, Kyoto 606-8502, Japan\\
\small e-mail address: {kusuoka@math.kyoto-u.ac.jp}}
\maketitle

\begin{abstract}
In the present paper, we give some examples of stochastic differential equations which have delicateness in the Markov and strong Markov properties, the uniqueness locally in time and globally in time, and initial conditions.
Moreover, we show that such stochastic differential equations appear in the limits of stochastic differential equations which have the existence and pathwise uniqueness of solutions.
These examples are constructed in motivation to singular stochastic partial differential equations.
We also give some examples of shifted equations whose sum of the solutions depends on the choices of the decomposition of the initial condition of the original equation.
\end{abstract}

{\bf AMS Classification Numbers:} 60J60, 60H17, 60H10, 58J65

 \vskip0.2cm


\section{Introduction}\label{sec:intro}

After stochastic differential equations (SDEs) were introduced by Kiyoshi It{\^o}, SDEs had been studied very much as the mainstream in the probability theory.
The relations to other fields of mathematics also had been discovered, and the framework of SDEs are almost fixed now except some new types of SDEs (see \cite{FlRuWo1, FlRuWo2, FlIsRu} for an extension of the framework of SDEs, and \cite{GuPe} for a review of the history from the extension of SDEs to singular SPDEs).
As partial differential equations (PDEs) including some noises, stochastic partial differential equations (SPDEs) also have been studied intensively, not only in view of the motivations of mathematics, but also of the motivations in physics.
We remark that SPDEs can be regarded as the infinite-dimensional version of SDEs.

Recently singular SPDEs are studied intensively.
The key methods for the study are regularity structures (see \cite{Ha1, Ha2}) and paracontrolled calculus (see \cite{GIP}).
The methods enable us to treat SPDEs which require renomalization to be solved in suitable sense.
Such equations appear when we consider stochastic quantization equations of quantum field theories (see the introductions of \cite{AlKu1, AlKu2} and \cite{GuHo2} for the history of stochastic quantizations).
Because of explicit applications, now singular SPDEs are studied by many mathematicians and mathematical physicists from both the mathematical and the physical points of view.

However, the method of singular SPDEs is very different from the standard theory of SDEs.
Indeed, to solve singular SPDEs we first consider the linearized SPDE associated to the original singular SPDE.
The linearized SPDE is often of the infinite-dimensional Ornstein-Uhlenbeck type, and hence the solution is often the Ornstein-Uhlenbeck process.
Second, we consider a transform of the singular SPDE to another SPDE which the difference between the solutions to the original singular SPDE and the linearized SPDE is expected to satisfy.
We call the equation which the difference of the solutions satisfies {\it a shifted equation}.
For the shifted equation to be well-posed we often need renormalization, which cancels the diverging terms coming from the nonlinearity in a suitable sense.
Then, we give a solution to the original singular SPDE in a suitable sense by the sum of the solutions to the shifted equation and the linearized equation.
This method is originally introduced in \cite{DPDe}.
Here, we remark that the definition of the solution to singular SPDEs are different from those to usual SPDEs.
There is another way to understand the solution to singular SPDEs.
Before taking the difference between the solutions to the original singular SPDE and the linearized SPDE, by regularizing the noise (or the nonlinear term) of singular SPDEs we can define the solution to the regularized singular SPDE in usual sense of solutions to SPDEs.
The solutions to the shifted equations of the regularized singular SPDE are obtained, and we can show the convergence of the solutions as removing regularization if the shifted equation is solved without regularization.
Then, the sum of the solutions to the shifted equations and the linearized SPDE is also convergent in a suitable sense.
Hence, the solution to the singular SPDE can be regarded as the limit of the solutions to the regularized equations.
Here, we remark that the limits are known to be independent of regularizations in suitable classes (see \cite{Ha1} and \cite{GIP}).
See Section \ref{sec:SSPDE} for more details of singular SPDEs.

The points that we have to be careful in the methods of singular SPDEs are the following.
\begin{itemize}
\item The solutions to singular SPDEs are defined by limits of solutions to approximate equations, or defined by a projection of solutions to shifted equations.

\item There exist exceptional sets of initial points of original equations and shifted equations.
\end{itemize}
These do not appear in the standard theory of SDEs, and seem to make delicate problems differently from normal SDEs.
See Section \ref{sec:Preliminary} for the details of the methods of singular SPDEs and the standard theory of SDEs.

From these points of view, in the present paper we consider examples under the framework of SDEs which make phenomena different from the cases of sufficiently nice SDEs.
Precisely, we focus on the importance of the uniqueness in all initial conditions and the delicateness of taking limits of strong Markov processes.

Now we summarize the results in the present paper as the following theorem.

\begin{thm}\label{thm:main}
\begin{enumerate}
\item There exists a sequence of SDEs which have the pathwise uniqueness, such that their solutions have the strong Markov property and converge to a stochastic process which does not have the strong Markov property. (See Theorem \ref{thm:refl}.)

\item There exists a SDE which has the uniqueness locally in time for every initial point except one point, whose any solution is extendable globally in time.
However, the SDE does not have the uniqueness globally in time. (See Theorem \ref{thm:nuSDE}.)

\item There exist sequences of SDEs which have the pathwise uniqueness for every initial point, such that the solutions to the both sequences of SDEs converges to solutions to a SDE which does not have the uniqueness of solutions.
Moreover, the limits of the solutions are different from each other. (See Theorem \ref{thm:nuSDEapprox}.)

\item There exists a path-dependent SDE which has the uniqueness locally in time for every initial point except one point, whose any solution is extendable globally in time.
However, the SDE does not have the uniqueness globally in time. (See Theorem \ref{thm:nuSDEpdep}.)

\item There exists a sequence of SDEs which have the pathwise uniqueness for every initial point, such that the solutions to the both sequences of SDEs converges to solutions to a SDE which does not have the uniqueness of solutions.
Moreover, the limit of the sequence of SDEs can be regarded a shifted equation.
However, the limit of the solutions depends on the decomposition of initial conditions for the shifted equation. (See Theorem \ref{thm:shift3approx}.)
\end{enumerate}
\end{thm}

In view of the examples in Theorem \ref{thm:main} we see the difference of the frameworks between singular SPDEs and the standard theory of SDEs.
Moreover, arguments in some papers in the field of singular SPDEs seem loose.
We dare not to refer them here in the present paper.
It should be emphasized that the results in the present paper do not conflict with known results in singular SPDEs, but suggest careful arguments to apply known facts in conventional stochastic analysis.
Here, we also remark that in the case of singular SPDEs on the Euclidean spaces, we need to take limits not only for regularization, but also for infinite volumes (see \cite{MW2} for the $\Phi^4$-stochastic quantization on ${\mathbb R}^2$ and \cite{AlKu2, GuHo1, GuHo2} for that on ${\mathbb R}^3$). Hence, also with respect to the limits for the infinite volumes, we have to pay attentions in view of such examples in Theorem \ref{thm:main}.

On the other hand, some papers are written very carefully.
In the first memorable paper \cite{Ha1} of the theory of regularity structures, the uniqueness of the solution is discussed in the framework of the regularity structures (see \cite[Theorem 7.8]{Ha1}), while in applications to singular SPDEs the author did not mention the uniqueness of the solution, but mentioned only the independence of regularizations in a suitable class is mentioned (see \cite[Section 1.5]{Ha1}).
Also in the first memorable paper \cite{GIP} of paracontrolled calculus, the dependence of inputs and the sense of the unique solutions are clarified (see \cite[Theorem 3.3, Theorem 4.1, Corollary 5.9]{GIP}).
By the way, the continuities in the initial conditions in suitable senses are obtained in both theories (see \cite{Ha1} and \cite{GIP}), and from the continuities we may prove that many singular SPDEs do not make problems in Theorem \ref{thm:main}.

Finally we remark that in the cases of the $\Phi ^4_2$-stochastic quantization and the $\exp (\Phi)_2$-stochastic quantization, the stochastic quantization equations are explicitly written by means of Wick products.
This means that the singular SPDEs are well-defined without taking limits.
Moreover, we are able to construct the associated Dirichlet forms, which generate the strong Markov processes.
See \cite{AlRo, RZZ} for the case of the $\Phi ^4_2$-model and \cite{HKK1, HKK2} for the case of the $\exp (\Phi)_2$-model.
In view of these facts, the author in the present paper guesses that singular SPDEs which have sufficiently low singularity of noises do not make delicate problems.
It also should be remarked that even these cases exceptional sets appear, because Wick products are defined almost everywhere with respect to the free field measure, and the Markov processes constructed by Dirichlet forms are defined only for quasi-every initial points with respect to the reference measure.

The construction of the present paper is as follows.
In Section \ref{sec:Preliminary} we recall the method of singular SPDEs via the $\Phi ^4$-stochastic quantization model, and the standard theory of SDEs.
Precisely, in Section \ref{sec:SSPDE} we give a short review of the methods of singular SPDEs to clarify the techniques to construct solutions to singular SPDEs.
And in Section \ref{sec:PreSDE} we recall the standard theory of SDEs to clarify the differences between them.
In Section \ref{sec:Markov} we see that for the strong Markov property, the transition probability should be defined at every initial points in the state space.
Precisely, we give an example of a stochastic process that the process is Markov, but not strong Markov, because of a problematic point in the state space.
Moreover, we show that such a situation occurs even in the limit of the strong Markov processes which are the solutions to pathwise-unique SDEs.
This implies that we need to be careful in the strong Markov property of the limit processes.
In Section \ref{sec:nuSDE} we give an example of a SDE so that the solutions have the uniqueness locally in time for every initial points except one point, any solution is extendable to one globally in time, but the SDE does not have the uniqueness globally in time.
Moreover, we also give examples that such a situation can occur in the case of a limit of SDEs which have the pathwise uniqueness for all initial points, and also in the case of a path-dependent SDE with the pathwise uniqueness locally in time for all initial points. 
In Section \ref{sec:shift} we see a delicateness for the implication from the uniqueness of shifted equations to the uniqueness of the original equation.
In particular, we give the dependence of the choice of initial conditions of shifted equations on the solutions to the original equation.

\vspace{5mm}\noindent
{\bf Acknowledgements.}
The author is grateful to Emeritus Professors Yozo Tamura and Tokuzo Shiga for giving me many knowledges and techniques while the author was a student.
Indeed, the proofs in the present paper are based on the knowledges and the techniques given by them many years ago.
This work was partially supported by JSPS KAKENHI Grant Numbers 21H00988.

\vspace{5mm}\noindent
{\bf Notation.}
Denote $a\wedge b := \min \{ a,b \}$ and $a\vee b := \max \{ a,b \}$ for $a,b \in {\mathbb R}$.
We define the topology of the space $C([0,\infty )) = C([0,\infty ); {\mathbb R})$ by the uniform convergence on each compact set.
We denote by $P_x$ the law of some Markov processes with respect to initial point $x$, and by $E_x$ the expectation with respect to $P_x$.
Let $C_0({\mathbb R})$ be the total set of continuous functions on ${\mathbb R}$ which have compact supports.

\section{Preliminary}\label{sec:Preliminary}

\subsection{The technique in singular SPDEs}\label{sec:SSPDE}

Here, we see the motivation of the present paper by a short review of the technique in singular SPDEs.
As an explicit example, consider the stochastic quantization equation of the $\Phi ^4$-quantum field model 
\begin{equation}\label{eq:Phi4SQ1}
\partial _t \Phi (t,x) = (\triangle -1) \Phi (t,x) - \Phi(t,x)^3 + \dot{W}(t,x) , \quad (t,x)\in [0,\infty )\times {\mathbb T}^d
\end{equation}
where $\dot W(t,x)$ is the time-space white noise (often regarded as weak derivative in time $t$ of the cylindrical Brownian motion $W(t,x)$ on $L^2({\mathbb T}^d)$).
In the case that the dimension $d=1$, the singularity of $\dot W$ is not high and \eqref{eq:Phi4SQ1} is solved by usual methods in SPDEs.
However, when $d \geq 2$, the singularity of $\dot W$ is too high to solve \eqref{eq:Phi4SQ1} without renormalization.
The problem is on the nonlinear term $\Phi(t,x)^3$, which is not well-defined as it is, if $\Phi$ is not a function, but a distribution.
Indeed, \eqref{eq:Phi4SQ1} is ill-defined in the sense that the expected regularity of $\Phi$ is too low, precisely $\Phi$ is expected to be a distribution, and $\Phi ^3$ is not well-defined as it is.

In the case that $d=2$, \eqref{eq:Phi4SQ1} is modified by renormalization as
\begin{equation}\label{eq:Phi4SQ2}
\partial _t \Phi (t,x) = (\triangle -1) \Phi (t,x) - :\Phi(t,x)^3: + \dot{W}(t,x) , \quad (t,x)\in [0,\infty )\times {\mathbb T}^d
\end{equation}
where $:\phi ^n:$ is the wick product with respect to the free field measure for $n\in {\mathbb N}$.
Here, we note that the map $\phi \mapsto :\phi ^n:$ is well-defined for almost every $\phi$ with respect to the free field measure.
A direct approach by methods of SPDEs to \eqref{eq:Phi4SQ2} was first introduced in \cite{DPDe}.
Now we see the strategy of the approach, because it is closely related to the motivation of the present paper.
Consider the stationary solution to the linearized equation of \eqref{eq:Phi4SQ2}:
\begin{equation}\label{eq:Phi4OU}
\partial _t Z(t,x) = (\triangle -1) Z(t,x) + \dot{W}(t,x), \quad (t,x)\in [0,\infty )\times {\mathbb T}^d,
\end{equation}
which is given by
\[
Z(t) := \int _{-\infty}^t e^{(t-s)(\triangle -1)} dW(s), \quad t\in [0,\infty ),
\]
where $(W(t); t\in (-\infty ,\infty ))$ is an extension in the time parameter $t$ of the cylindrical Brownian motion $(W(t); t\in [0,\infty ))$ which satisfies $\partial _t W(t) = \dot{W}(t,\cdot)$.
If $\Phi$ and $Z$ are solutions to \eqref{eq:Phi4SQ2} and \eqref{eq:Phi4OU} in a suitable sense, respectively, then $Y:= \Phi - Z$ satisfies
\begin{equation}\label{eq:Phi4Shifted}\begin{array}{rl}
\partial _t Y(t,x) &= (\triangle -1) Y(t,x) - Y(t,x)^3 -3 Z(t,x) Y(t,x)^2 \\
&\hspace{2cm} -3 :Z(t,x)^2: Y(t,x) - :Z(t,x)^3:, \quad (t,x)\in [0,\infty )\times {\mathbb T}^2.
\end{array}\end{equation}
We often call \eqref{eq:Phi4Shifted} a shifted equation.
Since the coefficients $Z(t,x)$, $:Z(t,x)^2:$ and $:Z(t,x)^3:$ of this equations are explicit and have better regularities than the original noise $\dot{W}(t,x)$, \eqref{eq:Phi4Shifted} is solved in the sense of mild solutions.
Hence, by letting $\Phi := Y+Z$ we have a solution to \eqref{eq:Phi4SQ2} (recall that $Z$ is an explicit stochastic process).
This is the strategy of the argument in \cite{DPDe}.
Here, it should be remarked that $\Phi := Y+Z$ is the definition of the solution to \eqref{eq:Phi4SQ2} (see \cite[Definition 4.1]{DPDe}), and the renormalized equation \eqref{eq:Phi4SQ2} is not solved directly.
Moreover, there is ambiguity on the initial condition $\Phi (0,x) = Y(0,x)+Z(0,x)$, because $Z$ is a stationary solution to \eqref{eq:Phi4OU} and has exceptional sets with respect to the probability measure.
Furthermore, for given initial condition $\Phi (0,x)$ there exist infinitely many choices to decompose it into $Y(0,x)$ and $Z(0,x)$ by $\Phi (0,x) = Y(0,x)+Z(0,x)$, and the different decomposition can give different solutions $\Phi$.
These facts make difficulties to discuss the uniqueness of solutions.
The uniqueness is closely related to the strong Markov property of solutions (see Section \ref{sec:PreSDE} for a brief review of the standard theory of SDEs, and Section \ref{sec:Markov} for examples of non-strong Markov solutions).
The motivation of the present paper is to study the differences between the solutions obtained by this approach and by the standard theory of SDEs.
Precisely, we give some explicit examples which make different situations from the standard theory of SDEs.

It also should be remarked that the Dirichlet form associated to \eqref{eq:Phi4SQ1} had been constructed in \cite{AlRo} before \cite{DPDe} appeared.
The advantage of the method of the Dirichlet form is that we obtain the unique Hunt process associated to the bilinear form, which has the strong Markov property.
The solution to \eqref{eq:Phi4SQ2} is also constructed by \cite{AlRo}.
However, there also exists exceptional sets, which are polar sets of the Hunt process, and a similar situation occurs to the approach by \cite{DPDe}.
See \cite{FOT} for the theory of the Dirichlet forms.

Now we turn to the case that $d=3$.
In this case, the singularity of the noise $\dot{W}(t,x)$ is higher than the case that $d=2$, and solving \eqref{eq:Phi4SQ1} via renormalization had been an open problem for many years.
However, recently it became solvable in the sense of singular SPDEs by the regularity structure introduced in \cite{Ha1} and the paracontrolled calculus in \cite{GIP}.
Here, we remark that the regularity structure and the paracontrolled calculus are different theories, but share the key idea that similarly to \cite{DPDe} we decompose the formal solution $\Phi$ by grades of regularities (singularities), and transform the original SPDE \eqref{eq:Phi4SQ1} into solvable one.
Indeed, the method in \cite{DPDe} can be regarded as the most simple case of singular SPDEs.
In the regularity structure the pair of the decomposed terms of $\Phi$ is regarded as an element the model space, the solution $\Phi$ to the renormalized version of \eqref{eq:Phi4SQ1} is constructed by the projection of the solution in the model space to the regularity structure associated to \eqref{eq:Phi4SQ1}.
See \cite{Ha2} for a survey of the approach by the regularity structure to the $\Phi ^4_3$-stochastic quantization equation, and \cite{BCCH} for the regularity structure approach to other singular SPDEs.

In the paracontrolled calculus, the formal solution $\Phi$ is decomposed by the paracontrolled ansatz, which is also a decomposition by the regularities (singularities).
For \eqref{eq:Phi4SQ1} the paracontrolled ansatz is given by
\begin{equation}\label{eq:Phi4SQ3sol}
\Phi = Z + I(Z^3) - 3 \int _{-\infty}^t e^{(t-s)(\triangle -1)} \left[ (\Psi _1 + \Psi _2 - I( :Z^3: )) \mbox{\textcircled{\scriptsize$<$}} :Z^2: \right] (s) ds + \Psi _2
\end{equation}
where $I( :Z^3: )(t) = \int _{-\infty}^t e^{(t-s)(\triangle -1)} :Z^3(s): ds$ and $\mbox{\textcircled{\scriptsize$<$}}$ is a notation of the paraproduct.
By solving nonlinear partial differential equation of $(\Phi _1 , \Phi _2)$:
\begin{equation}\label{eq:Phi4SQ3shift}\left\{ \begin{array}{rl}
\partial _t \Psi _1 (t)&= (\triangle -1) \Psi _1(t) - 3\left[ \Psi _1 + \Psi _2 - I( :Z^3: )) \mbox{\textcircled{\scriptsize$<$}} :Z^2: \right] (t) \\
\partial _t \Psi _2 (t)&= (\triangle -1) \Psi _2(t) + F(\Psi _1, \Psi _2, Z, :Z^2:, I( :Z^3: )) (t),
\end{array}\right.\end{equation}
which is obtained by a transformation of \eqref{eq:Phi4SQ1} together with renormalization, we can define a solution $\Phi$ by \eqref{eq:Phi4SQ3sol} to the renormalized version of \eqref{eq:Phi4SQ1}.
The term $F(\Psi _1, \Psi _2, Z, :Z^2:, I( :Z^3: ))$ is too complicated to write down explicitly, because it includes paraproducts, resonance, and their commutators.
But here, we remark that $F(\Psi _1, \Psi _2, Z, :Z^2:, I( :Z^3: ))(t)$ depends on the trajectory of $Z$ before the present time $t$, i.e. \eqref{eq:Phi4SQ3shift} is a partial differential equation of the path-dependent type.
Indeed, $F$ depends on $I( :Z^3: )$.
Hence, even if \eqref{eq:Phi4SQ3shift} has the uniqueness of the solution $(\Phi _1 , \Phi _2)$ for all initial conditions in a suitable state space, it is not trivial for $(\Phi _1 , \Phi _2)$ to have the Markov and strong Markov properties.
For the details of this argument, see \cite{CaCu}, \cite{MW3} or \cite{AlKu1}.
The definition of the solution $\Phi$ given by \eqref{eq:Phi4SQ3sol} is same as in the case that $d=2$ (see above or \cite[Definition 4.1]{DPDe} for the precise definition).
The renormalized equation of \eqref{eq:Phi4SQ1} for $d=3$ is written formally by
\begin{equation}\label{eq:Phi4SQ3}
\partial _t \Phi (t,x) = (\triangle -1) \Phi (t,x) - (\Phi(t,x)^3 - \infty \cdot \Phi (t,x)) + \dot{W}(t,x) , \quad (t,x)\in [0,\infty )\times {\mathbb T}^3.
\end{equation}
We remark that the term $(\Phi(t,x)^3 - \infty \cdot \Phi (t,x))$ cannot be written by Wick products, because we need further renormailzation than Wick products.
This means that we do not have explicit representation of the term, and it is jutisfied by limits, i.e. approximations by regularization.
It is also remarked that the renormalization constants are written explicitly, once we fix an approximation.
For the renormalization constants to be independent of $t$, we need a restriction on the initial condition $Z_0$ of $Z$, and hence there is anbiguity of the initial condition of $\Phi$ similar to the case that $d=2$.
We enphasize here that the target \eqref{eq:Phi4SQ3} is a formal equation, and the solution to \eqref{eq:Phi4SQ3} is defined by the solution to a path-dependent equation \eqref{eq:Phi4SQ3shift} via approximations.

Now shortly, we mention a few references of the approach by the paracontrolled calculus to the $\Phi ^4_3$-stochastic quantization equation.
The approach of the paracontrolled calculus to the $\Phi ^4_3$-stochastic quantization equation was first introduced in \cite{CaCu}.
The solution obtained in \cite{CaCu} is locally in time, and unique in the sense of the paracontrolled calculus.
In \cite{MW3} the extendability globally in time of the solutions to \eqref{eq:Phi4SQ3sol} was obtained for sufficiently regular initial conditions of $(\Psi _1, \Psi _2)$.
Here, we remark that the global exisntence and the local uniqueness of the solution imply neither the global uniquness nor the strong Markov property of solutions, as we will see in the following sections.
On the other hand, in \cite{AlKu1} the solution (in the sense of the paracontrolled calculus) globally in time for almost every initial point with respect to the $\Phi ^4_3$-measure was obtained by stationary approximation.
This approach by stationary approximation enables us to constuct a time-global solution for almost every initial points with respect to the $\Phi ^4_3$-measure which is constructed by the approach itself.

We remark that the regularity structure and the paracontrolled calculus are explained above independently, but the precise relation between the regularity structure and the paracontrolled calculus has been studied in \cite{BaHo1, BaHo2}.

What we would like to emphasize here is that the methods in singular SPDEs are very different from the standard theory of SDEs.
To clarify the difference we recall the standard theory of SDEs in Section \ref{sec:PreSDE}.

\subsection{Markov processes and SDEs}\label{sec:PreSDE}

In this section we recall the standard theories of Markov processes and SDEs.

First, we recall the key theorem to construct the strong Makov processes form Markov semigroups.
For a locally compact Hausdorff space $S$ with the second axiom of countability, let $C_{\infty}(S)$ be the total sets of continuous functions $f$ such that $f$ is extended to a continuous function $\tilde{f}$ on $S\cup \{ \Delta \}$ which is the one-point compactification of $S$, and that $\tilde{f}(\Delta ) =0$.

\begin{thm}[cf. \cite{BlGe}]\label{thm:BlGe}
Let $\{ T_t\}$ be a nonnegative contractive and semigroup on $C_\infty (S)$.
Assume that $\{ T_t\}$ satisfies the Feller property, i.e. $\{ T_t\}$ is a strong continuous semigroup on $C_\infty (S)$.
Then, there exists a Hunt process which has $\{ T_t\}$ as its transition semigroup.
\end{thm}

See \cite[Section A.2]{FOT} for the definition of Hunt processes.
We remark that Hunt processes have the strong Markov property (see \cite[Theorem A.2.1]{FOT}).
The Feller property, in particular $\lim _{x\rightarrow \Delta} (T_t f)(x)=0$ for $f \in C_\infty (S)$, implies that, the approaching of the associated stochastic process to $\Delta$ is regarded as explosion.
In Section \ref{sec:Markov} we will see the importance of the Feller property, in particular of the assumption that the semigroup is defined on $C_{\infty}(S)$.

For the strong Markov property of solutions to (time-homogenious) SDEs, the following is known. 
\begin{center}\em
If a SDE has uniquness in law for all initial values\\
in the state space, then the solutions satisfy the strong Markov property.
\end{center}
For the precise statement and the proof, see e.g. \cite[Theorem 5.1, Chapter IV]{IW} and \cite[Theorem 6.2.2]{StVa}.
For notations, let $X$ be the solution to a SDE, $P_x$ the law of $X$ with respect to initial point $x$, $E_x$ the expectation with respect to $P_x$, $\{ {\mathcal F}_t\} _t$ the filtration for solutions, and let $B$ be the Brownian motion of the driving noise in the SDE.
The key points in the proof are the following facts.
\begin{enumerate}
\item[{(s1)}] For any finite stiopping time $\tau$, $t\mapsto B_{t+\tau}-B_{\tau}$ is again a Brownian motion with the filtration $\{ {\mathcal F}_{t+\tau} \} _t$.

\item[{(s2)}] The uniqueness for all initial value $x\in S$ implies the connection of solutions i.e.
\[
P_x( X_{\cdot \wedge \tau}\in A_1, X_{\cdot +\tau} \in A_2) = E_x[ {\mathbb I}_{A_1} (X_{\cdot \wedge \tau}) P_{X_\tau }(A_2)]
\]
for any $A_1, A_2 \in {\mathcal B}(C([0,\infty ); {\mathbb R}^n))$ and a finite stopping time $\tau$.
\end{enumerate}
The elementary fact (s1) is necessary. Unless (s1), the behaviour of $X$ after $\tau$ can be different from the law of $X$ starting at $X_{\tau}$, because of the difference of the law of the driving noises.
On the other hand, (s1) is hardly paid attention, because by the martingale representation theorem (see e.g. \cite[Section 7, Chapter II]{IW}) we can transform SDEs driven by semi-Martingales into those driven by Brownian motions and drift terms.
Thanks to (s1) and the uniqueness of the solution, we have (s2).
If the uniqueness does not hold for an initial point $x_0 \in S$ and $\{ x_0\}$ is not a polar set, then by letting $\tau$ be the hitting time at $x_0$ we can choose two different $P_{X_\tau}(A_2)$ for a $A_2 \in {\mathcal B}(C([0,\infty ); {\mathbb R}^n))$.
Here, we remark that in the argument of singular SPDEs the driving noise is the Ornstein-Uhlenbeck process rather than a white noise, and does not satisfy the property like (s1).
We also remark that for the strong Markov property the uniquness for all initial points in the state space is important.
Precisely, the acceptable exceptional sets are only polar sets.

In the theory of SDEs, if there exist some problematic points for initial conditions, we often remove the points from the state space and regard them as boundaries.
Now recall that the background of the present paper is on singular SPDEs, and that as in Section \ref{sec:SSPDE} we do not regard such points as boundaries.
Another way to treat SDEs on a state space $S$ which is not an Euclidean space, is regarding them as SDEs on a manifold $S$.
The solution to SDEs on manifolds are defined by connections of paths via consistency of the local solutions which are solved on local charts.
The consistency comes from the strong Markov property of the solutions, which is guaranteed from the uniqueness of the property on local charts.
As mentioned above, to obtain the strong Markov property we need the uniqueness of the solution first. 
See \cite[Chapter V]{IW} and \cite{El} for the details of SDEs on manifolds.

On the other hand, once some points are removed from the Euclidean spaces, the space after removed is not complete.
So, we have to regard the exit time $\zeta$ from $S$ as the explosion time (life time).
This is very reasonable in view of the facts above.
Indeed, regarding the exit time $\zeta$ from $S$ as the explosion time corresponds to fixing a boundary condition on $\partial S$, which is the killing boundary condition.
Moreover, once we put the killing boundary condition on $\partial S$, the processes starting at the problematic initial points in $\partial S$ are uniquely determined as the processes staying at the points.
From these arguments, it is reasonable to regard hitting at problematic points as explosion of the processes.
It is also remarked again that the Feller property means explosion at $\partial S$ for the processes as mentioned above.

In Section \ref{sec:Markov} we give some explicit examples of the stochastic processes which do not have the strong Markov property, but have the Markov property.
In Section \ref{sec:nuSDE} we see explicit examples of SDEs such that the existence and uniqueness hold locally in time except one point, but the global uniqueness does not hold.
These examples are constructed by providing a problematic point for initial conditions.
This fact implies the importance for stochastic processes to have the uniqueness in all initial points in the state space.
Moreover, in the same sections, we see that these phenomena can happen for the limits of SDE which have the existence and uniqueness of the solutions.
This implies that the delicateness is hidden by the limiting procedure.

As mentioned in Sections \ref{sec:intro} and \ref{sec:SSPDE}, the motivation of the present paper is not on SDEs, but on SPDEs.
However, since SPDEs can be regarded as infinite-dimensional versions of SDEs, the examples are sufficient to see the delicateness of the arguments in singular SPDEs.

\section{Markov processes without strong Markov properties}\label{sec:Markov}

In this section, we see examples of the Markov processes without strong Markov properties.

\begin{prop}\label{prop:nonMarkov1}
Let $S:= [-\pi , \pi]$ and $Y=(Y_t)$ be the Brownian motion on $S$ with reflecting boundary condition at $-\pi$ and $\pi$.
Let $S^\circ := (-\pi , \pi)$ and define a family of operators $\{ T_t^\circ; t\in [0,\infty )\}$ on $C_b (S^\circ)$ by
\[
T_t^\circ f(x) := E_x[f(Y_t); Y_t\in S^\circ], \quad x\in S^\circ ,\ f\in C_b(S^\circ ).
\]
Then, $\{ T_t^\circ\}$ is a nonnegative, contractive and conservative semigroup on $C_b (S^\circ)$.
Moreover, for $f\in C_\infty (S^\circ )$ it holds that
\[
\lim _{t\downarrow 0} \| T_t^\circ f  - f \| _{\infty} =0.
\]
\end{prop}

\begin{proof}
For each $t\in [0,\infty )$ and $x\in S^\circ$, $P_x(Y_t\in S\setminus S^\circ) =0$.
From this and the fact that $Y_t$ is a continuous Markov process on $S$, we have the assertion.
\end{proof}

\begin{rem}
It is easy to see that $T_t^\circ$ in Proposition \ref{prop:nonMarkov1} does not generate a continuous Markov process on $S^\circ$.
Indeed, the candidate of the continuous Markov process is $Y$ in Proposition \ref{prop:nonMarkov1}, and the state space of $Y$ cannot be $S^\circ$, but $S$.
The reason why Theorem \ref{thm:BlGe} is not applicable to $T_t^\circ$ is that
\[
\lim _{x\rightarrow \pm \pi} T_t^\circ f(x) = E_{\pm \pi}[f(Y_t)] \quad (\mbox{double-sign in same order}),
\]
and the right-hand side is not $0$ for some $f\in C_\infty (S^\circ)$.
\end{rem}

\begin{prop}\label{prop:nonMarkov2}
Define $Y_t$ as in Proposition \ref{prop:nonMarkov1} and let $Z_t:= (\cos Y_t, \sin Y_t)$ for $t\in [0,\infty )$.
Then, $Z$ is a continuous stochastic process on the unit circle $S^1$ in ${\mathbb R}^2$.
Let
\[
p_t^Z(x,dz) := P_y ( (\cos Y_t, \sin Y_t) \in dz), \quad x\in S^1\setminus (-1,0)
\]
where $y\in (-\pi ,\pi )$ such that $x= (\cos y, \sin y)$.
Then, it holds that
\begin{align}\label{eq:propnonMarkov2-01}
E_x[g(Z_s) f(Z_{t+s})] &= \int_{S^1\setminus (-1,0)} g(w) \left( \int _{S^1\setminus (-1,0)} f(z) p_t^Z(w,dz)\right) p_s^Z(x,dw) , \\
\nonumber&\hspace{2cm} f,g\in C(S^1),\ s,t\in [0,\infty ),\ x\in S^1\setminus (-1,0).
\end{align}
In particular, there exists a modification $\widetilde{Z}$ of $Z$ (i.e. $P(\widetilde{Z}_t=Z_t)=1$ for all $t\in [0,\infty )$) so that $\widetilde{Z}_t \in S^1\setminus (-1,0)$ for all $t\in [0,\infty )$, and all such modifications $\widetilde{Z}$ is a Markov process on $S^1\setminus (-1,0)$.
However, neither $Z$ nor $\widetilde{Z}$ satisfy the strong Markov property.
\end{prop}

\begin{proof}
Since $P_y(Y_s =\pm \pi) =0$ for all $y\in (-\pi ,\pi)$ and $t\in [0,\infty )$, it holds that
\[
E_y[f((\cos Y_t, \sin Y_t))] = \int_{S^1\setminus (-1,0)} f(z) p_t^Z(x,dz), \quad f\in L^\infty (S^1),
\]
where $y\in (-\pi ,\pi )$ and $x= (\cos y, \sin y) \in S^1\setminus (-1,0)$.
Hence, by the Markov property of $Y$, we have for $f,g\in C(S^1)$, $s, t\in [0,\infty )$ and $x\in S^1\setminus (-1,0)$
\begin{align*}
E_x[g(Z_s) f(Z_{t+s})] &= E_y[g((\cos Y_s, \sin Y_s)) f(\cos Y_{t+s}, \sin Y_{t+s})] \\
&= E_y\left[ g((\cos Y_s, \sin Y_s)) E_{(\cos Y_s, \sin Y_s)}[f(\cos Y_{t}, \sin Y_{t})] \right] \\
&= \int_{S^1\setminus (-1,0)} g(w) \left( \int _{S^1\setminus (-1,0)} f(z) p_t^Z(w,dz)\right) p_s^Z(x,dw) .
\end{align*}
Thus, \eqref{eq:propnonMarkov2-01} is obtained.

By considering the hitting time $\tau$ of $Z$ at $(-1,0)$, it is easy to prove that $Z$ does not satisfy the strong Markov property.
For $\widetilde{Z}$, consider the sequence of the hitting times $\tau _n$ of $\widetilde{Z}$ at $\{ z\in S^1\setminus \{ (-1,0)\} ; {\rm dist}(z,(-1,0) )< 1/n\}$ and the limit $\tau := \lim _{n\rightarrow \infty} \tau _n$.
Then, similarly to the case of $Z$, we can prove that $\widetilde{Z}$ does not satisfy the strong Markov property.
\end{proof}

\begin{rem}\label{rem:nonMarkov2}
\begin{itemize}
\item Under the setting in Proposition \ref{prop:nonMarkov2} $Y$ is a strong Markov process on $[-\pi , \pi]$, while $Z$ is a stochastic process on $S^1$ but not strong Markov.
Intuitively these facts seem natural, because $[-\pi , \pi ] \ni y \mapsto x\in S^1$ is not injective.

\item For $Z$ to be a strong Markov process we have to give a suitable boundary condition at $(-1,0)$, which is the reflecting boundary condition.
The reflecting boundary condition decompose the point $(-1,0)$ into the two points: the limits along $S^1$ from the upper-half and lower-half planes.

\item The construction of a strong Markov process from $\widetilde{Z}$ is more serious, because $\widetilde{Z}$ cannot be an almost-sure right-continuous process on $S^1\setminus (-1,0)$.
To obtain a right-continuous modification of $\widetilde{Z}$ we have to extend the state space $S^1\setminus (-1,0)$.
When we choose $S^1$ as the extension of $S^1\setminus (-1,0)$, the right-continuous modification of $\widetilde{Z}$ is $Z$ and we need the procedure above to construct a strong Markov process.
The best choice of the extended space of $S^1\setminus (-1,0)$ is a space $\bar{S}$ so that the limits along $S^1$ from the upper-half and lower-half planes are decoupled.
Then, by choosing a right-continuous modification of $\widetilde{Z}$ on $\bar{S}$ we obtain a strong Markov process.
\end{itemize}
\end{rem}

From Propositions \ref{prop:nonMarkov1} and \ref{prop:nonMarkov2} and Remark \ref{rem:nonMarkov2} we see that for the construction of the strong Markov processes, the state space should be chosen suitably so that the process would be right-continuous and the transition probability should be determined for all initial point in the state space.
On the other hand, in view of Theorem \ref{thm:BlGe} if we provide the one-point compactification of the state space and regard reaching of the stochastic process to the extra point as explosion of the process, we obtain a strong Markov process.
In the case of Proposition \ref{prop:nonMarkov2}, this procedure corresponds to giving the absorbing boundary condition (Dirichlet boundary condition) at $(-1,0)$ for $Z$ or $\widetilde{Z}$.

In the setting of Proposition \ref{prop:nonMarkov2}, $Z$ is a stochastic process on $S^1$ such that $Z$ behaves like the Brownian motion on $S^1\setminus (-1,0)$ and reflected at $(-1,0)$.
It is also possible to construct a stochastic process $Z^+$ so that $Z^+$ behaves like $Z$ on $S^1\setminus (-1,0)$, but $Z^+$ passes trough from the lower-half plane to the upper-half plane and is reflected when $Z$ comes from the upper-half plane.
The precise procedure of the construction is as follows.
Let $Y^{+,x}$ be the Brownian motion on $[-\pi ,\pi ]$ starting at $x$ with reflecting boundary condition at $\pi$ and trapped at $-\pi$, and let $\{ \widetilde{Y}^{+,i}; i\in {\mathbb N} \}$ be independent copies of $Y^{+,0}$.
Denote $\tau _{-\pi}^0$ and $\tau _{-\pi }^i$ be the first hitting times of $Y^{+,x}$ and $\widetilde{Y}^{+,i}$ at $-\pi$, respectively.
Define the stochastic process $Z^+$ by
\[
Z_t^{+,x} := \left\{ \begin{array}{l}
(\cos Y_t^{+,x}, \sin Y_t^{+,x}), \quad 0\leq t< \tau_{-\pi}^0, \\
\left( \cos \widetilde{Y}_{t-\tau_{-\pi}^0 + \tau_{-\pi}^1 + \cdots + \tau_{-\pi}^{k-1}}^{+,k}, \sin \widetilde{Y}_{t-\tau_{-\pi}^0 + \tau_{-\pi}^1 + \cdots +\tau_{-\pi}^{k-1}}^{+,k} \right),\\
\qquad \tau_{-\pi}^0 + \tau_{-\pi}^1 + \cdots +\tau_{-\pi}^{k-1} \leq t< \tau_{-\pi}^0 + \tau_{-\pi}^1 + \cdots +\tau_{-\pi}^{k},\ k=0,1,2, \dots
\end{array} \right. 
\]
We remark that $( Z_t^{+,x}; t\in [0,\infty ) )$ is well-defined as a continuous stochastic process, because $\lim _{k\rightarrow \infty} (\tau_{-\pi}^0 + \tau_{-\pi}^1 + \cdots +\tau_{-\pi}^{k})= \infty$ almost surely and
\begin{align*}
&\lim _{t\rightarrow \tau_{-\pi}^0 + \tau_{-\pi}^1 + \cdots \tau_{-\pi}^{k}} \left( \cos \widetilde{Y}_{t-\tau_{-\pi}^0 + \tau_{-\pi}^1 + \cdots+\tau_{-\pi}^{k-1}}^{+,k}, \sin \widetilde{Y}_{t-\tau_{-\pi}^0 + \tau_{-\pi}^1 + \cdots +\tau_{-\pi}^{k-1}}^{+,k} \right) \\
&\quad = (-1,0) = (\cos \widetilde{Y}_0^{+,k+1}, \sin \widetilde{Y}_0^{+,k+1}), \quad k=0,1,2, \dots .
\end{align*}
Then, it is easy to see the stochastic process $( Z_t^{+,x}; t\in [0,\infty ) )$ satisfies the properties mentioned above.
We remark that  $( Z_t^{+,x}; t\in [0,\infty ) )$ satisfies the strong Markov property, because the behavior of $Z^{+,x}$ after hitting $(-1,0)$ does not depend on where $Z^{+,x}$ comes from.

Next we give an example of a stochastic process without the strong Markov property which is constructed by limits of strong Markov processes.
The construction is given by the penalization method, and such a method has been given in \cite{Slo}.

\begin{thm}\label{thm:refl}
For $n\in {\mathbb N}$, let $\varphi \in C^1({\mathbb R})$ be a nonnegative function satisfying
\begin{align*}
&\varphi (x) = 0, \quad x\in (-\infty ,-1] \cup [1,\infty ), \\
&\varphi' (x) \geq 0, \quad x\in [-1,0], \\
&\varphi' (x) \leq 0, \quad x\in [0,1]\\
&\int _{[-1,1]} \varphi (x) dx =1.
\end{align*}
Define $\varphi _n(x)$ by $\varphi _n(x) := n \varphi (nx)$ for $x\in {\mathbb R}$.
Consider the solution $X^n$ to the SDE
\begin{equation}\label{eq:SDEnrefl}
dX^n_t = \sigma (X^n_t) dB_t + b(X^n_t) dt - \varphi _n' (X^n_t) dt
\end{equation}
where $\sigma$ and $b$ are Lipschitz continuous functions such that $\sigma (x)>0$ for all $x\in {\mathbb R}$.

Then, when $\{ X^n\}$ is the family of the solutions to \eqref{eq:SDEnrefl} with tight initial laws, then the laws of $\{ X^n\}$ in $C([0,\infty ) ; {\mathbb R})$ is tight.
Moreover, if $X$ is a limit of a subsequence of the laws of $X^n$, $X$ hits $0$ with positive probability, but does not cross the origin almost surely, and $\{ t\in [0,\infty ); X_t =0\}$ is a null set with respect to the Lebesgue measure almost surely.
In particular, $X$ is not a strong Markov process.
\end{thm}

\begin{proof}
First we show the tightness of the laws of $X^n$.
Let $T>0$, $\varepsilon >0$ and choose a nondecreasing smooth function $f_\varepsilon$ such that
\begin{align*}
f_\varepsilon(x)&= x, \quad x\in \left( -\infty , -\frac{\varepsilon}{4}\right] \cup \left[ \frac{\varepsilon}{4}, \infty \right) ,\\
f_\varepsilon(x)&=0, \quad x\in \left[ -\frac{\varepsilon}{8}, \frac{\varepsilon}{8}\right] .
\end{align*}
The fact that $|x-f_\varepsilon (x)| \leq \varepsilon /4$ for $x\in {\mathbb R}$ and the It\^o formula imply that for $h>0$
\begin{align*}
&P\left( \sup _{s,t \in [0,T]; |t-s|<h} |X^n_t - X^n_s| >\varepsilon \right) \\
&\leq P\left( \sup _{s,t \in [0,T]; |t-s|<h} |f_\varepsilon (X^n_t) - f_\varepsilon (X^n_s) | >\frac{\varepsilon}2 \right) \\
&\quad + P\left( \sup _{s,t \in [0,T]; |t-s|<h} \left| [X^n_t - f_\varepsilon (X^n_t)] - [X^n_s - f_\varepsilon (X^n_s)] \right| >\frac{\varepsilon}2 \right) \\
&\leq P\left( \sup _{s,t \in [0,T]; |t-s|<h} \left| \int _s^t f_\varepsilon '(X^n_u) \sigma (X^n_u) dB_u + \int _s^t f_\varepsilon '(X^n_u) b(X^n_u) du \right. \right. \\
&\quad \hspace{3cm} \left. \left. - \int _s^t f_\varepsilon '(X^n_u) \varphi _n' (X^n_u) du + \frac{1}{2} \int _s^t f_\varepsilon ''(X^n_u) \sigma (X^n_u)^2 du \right| >\frac{\varepsilon}2 \right) .
\end{align*}
Note that $f_\varepsilon '(x) \varphi _n'(x) =0$ for $x\in {\mathbb R}$ and $n>8/\varepsilon$.
This inequality and the Chebyshev inequality imply that for $n>8/\varepsilon$
\begin{align*}
&P\left( \sup _{s,t \in [0,T]; |t-s|<h} |X^n_t - X^n_s| >\varepsilon \right) \\
&\leq P\left( \sup _{s,t \in [0,T]; |t-s|<h} \left| \int _s^t f_\varepsilon '(X^n_u) \sigma (X^n_u) dB_u \right| >\frac{\varepsilon}4 \right) \\
&\quad + P\left( \sup _{s,t \in [0,T]; |t-s|<h} \left| \int _s^t f_\varepsilon '(X^n_u) b(X^n_u) du + \frac{1}{2} \int _s^t f_\varepsilon ''(X^n_u) \sigma (X^n_u)^2 du \right| >\frac{\varepsilon}4 \right) \\
&\leq P\left( \sup _{s,t \in [0,T]; |t-s|<h} \left| \int _s^t f_\varepsilon '(X^n_u) \sigma (X^n_u) dB_u \right| >\frac{\varepsilon}4 \right) \\
&\quad + \frac{16}{\varepsilon ^2}E\left[ \sup _{s,t \in [0,T]; |t-s|<h} \left| \int _s^t \left( f_\varepsilon '(X^n_u) b(X^n_u) + \frac{1}{2} f_\varepsilon ''(X^n_u) \sigma (X^n_u)^2 \right) du  \right| ^2 \right] \\
&\leq P\left( \sup _{s,t \in [0,T]; |t-s|<h} \left| \int _s^t f_\varepsilon '(X^n_u) \sigma (X^n_u) dB_u \right| >\frac{\varepsilon}4 \right) + \frac{16}{\varepsilon ^2} ( \| f_\varepsilon ''\| _\infty ^2 \| \sigma \| _\infty ^4 h + \| f_\varepsilon '\| _\infty ^2 \| b\| _\infty ^2 h^2).
\end{align*}
On the other hand, the Burkholder-Davis-Gundy inequality implies
\begin{align*}
E\left[  \left| \int _s^t f_\varepsilon '(X^n_u) \sigma (X^n_u) dB_u \right| ^4 \right]
&\leq C E\left[ \left( \int _s^t |f_\varepsilon '(X^n_u)|^2 \sigma (X^n_u)^2 du \right) ^2 \right] \\
&\leq C \| f_\varepsilon '\| _\infty ^4 \| \sigma \| _\infty ^4 (t-s)^2 .
\end{align*}
As in the proof of \cite[Theorem 4.3 of Chapter I]{IW}, this estimate yields
\[
\lim _{h\downarrow 0} \sup _n P\left( \sup _{s,t \in [0,T]; |t-s|<h} \left| \int _s^t f_\varepsilon '(X^n_u) \sigma (X^n_u) dB_u \right| >\frac{\varepsilon}4 \right) =0 .
\]
Thus, we obtain
\[
\lim _{h\downarrow 0} \sup _n P\left( \sup _{s,t \in [0,T]; |t-s|<h} |X^n_t - X^n_s| >\varepsilon \right) = 0.
\]
This convergence and the tightness of the initial laws of $X^n$ yield the tightness of the laws of $X^n$ in $C([0,T]; {\mathbb R})$ (see \cite[Theorem 4.2 of Chapter I]{IW}).

Next we show that $X$ hits $0$ with positive probability, but does not cross the origin almost surely.
Let
\[
L := \frac{1}{2}\sigma (x)^2 \frac{d^2}{dx^2} + b(x)\frac{d}{dx}
\]
and let $l,r \in {\mathbb R}$ such that $l<r$.
Define a function $s_n$ on ${\mathbb R}$ by
\[
s_n(x) := \int _l^x \exp \left( -\int _l^y \frac{2(b(z)-\varphi _n'(z))}{\sigma (z)^2} dz \right) dy.
\]
Then, we have
\begin{equation}\label{eq:proprefl01}
\left( L- \varphi _n(x)\frac{d}{dx}\right) s_n(x) =0.
\end{equation}
Let $X^{n,x}_t$ be the solution starting at $x \in (-1,1)$, and let $\tau_c$ be the first hitting time of $X^{n,x}_t$ at $c$.
The general theory of one-dimensional diffusion processes and \eqref{eq:proprefl01} imply that for $l<x<r$
\[
P(X^{n,x}_{\tau_l\wedge \tau_r}=r) = \frac{s_n(x)-s_n(l)}{s_n(r)-s_n(l)} = \frac{s_n(x)}{s_n(r)}.
\]
Note that $\varphi (0)>0$.
In the case that $l<0<x<r$, an explicit calculation implies that
\begin{align*}
s_n(x)&\geq \int _l^0 \exp \left( -\int _l^y \frac{2(b(z)-\varphi _n'(z))}{\sigma (z)^2} dz \right) dy\\
&\geq \int _l^0 \exp \left( \left( \inf _{z\in [-(1/n),0]} \frac{1}{\sigma (z)^2} \right) \int _l^y \varphi _n' (z) dz -\int _{l}^y \frac{2b(z)}{\sigma (z)^2} dz \right) dy\\
&= \int _l^0 \exp \left( \left( \inf _{z\in [-(1/n),0]} \frac{1}{\sigma (z)^2} \right) \left( \varphi _n (y) -\varphi _n (l) \right) -\int _{l}^y \frac{2b(z)}{\sigma (z)^2} dz \right) dy\\
&= \int _{l/n}^0 \exp \left( \left( \inf _{z\in [-(1/n),0]} \frac{1}{\sigma (z)^2} \right) \cdot n \left( \varphi (\tilde{y}) -\varphi (nl) \right) -\int _{l}^{\tilde{y}/n} \frac{2b(z)}{\sigma (z)^2} dz \right) d\tilde{y}\\
&\rightarrow \infty \qquad \mbox{as}\ n\rightarrow \infty .
\end{align*}
Similarly, in the case that $l<x<0<r$,
\begin{align*}
s_n(x)&\leq \int _l^x \exp \left( \left( \sup _{z\in [-(1/n),0]} \frac{1}{\sigma (z)^2} \right) \int _l^y \varphi _n' (z) dz -\int _l^y \frac{2b(z)}{\sigma (z)^2} dz \right) dy\\
&\rightarrow \int _l^x \exp \left( -\int _l^y \frac{2b(z)}{\sigma (z)^2} dz \right) dy <\infty \qquad \mbox{as}\ n\rightarrow \infty .
\end{align*}
These yield
\begin{equation}\label{eq:proprefl02}
\lim _{n\rightarrow \infty} P(X^{n,x}_{\tau _l \wedge \tau_r}=r) =0, \quad l<x<0<r.
\end{equation}
Moreover, the fact that $P(X^{n,x}_{\tau_l \wedge \tau_r}=l) + P(X^{n,x}_{\tau_l \wedge \tau_r}=r) =1$ for all $n$ implies that
\begin{equation}\label{eq:proprefl03}
\lim _{n\rightarrow \infty} P(X^{n,x}_{\tau_l \wedge \tau_r}=l) =1, \quad l<x<0<r.
\end{equation}
Similarly, for $l<x<r<0$ we also have
\[
\lim _{n\rightarrow \infty} P(X^{n,x}_{\tau_l \wedge \tau_r}=r) = \lim _{n\rightarrow \infty} \frac{s_n(x)}{s_n(r)} = \frac{\int _l^x \exp \left( -\int _l^y \frac{2b(z)}{\sigma (z)^2} dz \right) dy}{\int _l^r \exp \left( -\int _l^y \frac{2b(z)}{\sigma (z)^2} dz \right) dy}.
\]
This implies that for $l<x<0$ there exists a constant $c_{l,x} \in (0,1)$ such that
\begin{equation}\label{eq:proprefl04}
\lim _{r\uparrow 0} \lim _{n\rightarrow \infty} P(X^{n,x}_{\tau_l \wedge \tau_r}=r) = c_{l,x} .
\end{equation}
From \eqref{eq:proprefl02}, \eqref{eq:proprefl03} and \eqref{eq:proprefl04} we see that the limit process $X$ of $X^n$ in law does not cross the origin from $\{ x<0\}$ to $\{ x>0\}$, and $X$ starting at $x<0$ hits $0$ with positive probability.
The proofs that $X$ does not cross the origin from $\{ x>0\}$ to $\{ x<0\}$ and $X$ starting at $x<0$ hits $0$ with positive probability, are similar to the argument above.
So we omit them.

Finally, we see that the limit process $X$ is not trapped at $0$.
For the proof it is sufficient to show that
\begin{equation}\label{eq:proprefl05}
\lim _{\varepsilon \downarrow 0} E\left[ \int _0^T {\mathbb I}_{[-\varepsilon ,\varepsilon]} (X_t) dt \right] =0
\end{equation}
for $T>0$ under the condition $\sup _{n\in {\mathbb N}} E[|X_0^n|^2]< \infty$.
Fix $\varepsilon \in (0,1)$ and $T>0$.
Choose a sequence $\{ g_m\} \subset C_0({\mathbb R})$ so that ${\mathbb I}_{[-\varepsilon ,\varepsilon]} \leq g_m \leq {\mathbb I}_{[-1,1]}$ for $m\in {\mathbb N}$ and $g_m (x) \downarrow {\mathbb I}_{[-\varepsilon ,\varepsilon]}(x)$ as $m\in {\mathbb N}$ for each $x\in {\mathbb R}$, and define $G_m \in C^2({\mathbb R})$ by
\[
G_m (x) := \int _0^x \left( \int _0^y g_m (z) dz\right) dy .
\]
Then, the It\^o formula implies
\begin{align*}
&E[G_m (X^n_T)] - E[G_m (X^n_0)]\\
&= \frac{1}{2}E\left[ \int _0^T g_m (X^n_t) \sigma (X^n_t)^2 dt \right] + E\left[ \int _0^T G_m'(X^n_t) b(X^n_t) dt \right] - E\left[ \int _0^T G_m'(X^n_t) \varphi _n' (X^n_t) dt \right] .
\end{align*}
Noting that the definition of $G_m$ yields that 
\[
G_m'(x) \varphi _n' (x) \leq 0, \quad x\in {\mathbb R},
\] 
we have
\[
E\left[ \int _0^T g_m(X^n_t) \sigma (X^n_t)^2 dt \right] \leq 2E[G_m(X^n_T)] - 2E[G_m(X^n_0)] - 2E\left[ \int _0^T G_m'(X^n_t) b(X^n_t) dt \right] .
\]
By the convergence of $X^n$ to $X$ in law on $C([0,\infty ))$, it follows that
\[
E\left[ \int _0^T g_m(X_t) \sigma (X_t)^2 dt \right] \leq 2E[G_m(X_T)] - 2E[G_m(X_0)] - 2E\left[ \int _0^T G_m'(X_t) b(X_t) dt \right] .
\]
Here, we remark that $|G_m(x)| \leq |x|$ for $x\in {\mathbb R}$ and that $\sup _{t\in [0,T]} E[ |X_t^n|^2] < \infty$ follows from application of the It\^o formula to $f(X^n_t)$ for some $f\in C^2({\mathbb R})$ such that 
\[ \left\{ \begin{array}{l}
x^2 \leq f(x) \leq 2(1+x^2), \quad x\in {\mathbb R},\\
f(x)=1, \quad x\in [-1,1],
\end{array}\right. \]
and a standard argument (c.f. e.g. \cite[Theorem 2.4, Chapter IV]{IW}).
Since
\begin{align*}
\lim _{m\rightarrow \infty} G_m'(x) &= \int _0^x {\mathbb I}_{[-\varepsilon ,\varepsilon]} (z) dz = x{\mathbb I}_{(-\varepsilon , \varepsilon )} (x) + \varepsilon {\mathbb I}_{[\varepsilon , \infty )} (x) - \varepsilon {\mathbb I}_{(-\infty , -\varepsilon ]} (x), \\
\lim _{m\rightarrow \infty} G_m(x) &= \int _0^x \left( \int _0^y {\mathbb I}_{[-\varepsilon ,\varepsilon]} (z) dz\right) dy = \frac{x^2}{2} {\mathbb I}_{(-\varepsilon ,\varepsilon )} (x) + \varepsilon \left( |x|- \frac{\varepsilon}{2}\right) {\mathbb I}_{(-\infty , -\varepsilon ] \cup [\varepsilon , \infty )}(x),
\end{align*}
by the dominate convergence theorem we have
\begin{align*}
&E\left[ \int _0^T {\mathbb I}_{[-\varepsilon ,\varepsilon]} (X_t) \sigma (X_t)^2 dt \right] \\
&\leq 2E\left[ \frac{X_T^2}{2} {\mathbb I}_{(-\varepsilon ,\varepsilon )} (X_T) + \varepsilon |X_T|\right] + 2E\left[ \frac{X_0^2}{2} {\mathbb I}_{(-\varepsilon ,\varepsilon )} (X_0) + \varepsilon |X_0| \right] \\
&\quad + 2\varepsilon E\left[ \int _0^T |b(X_t)| dt \right] .
\end{align*}
Therefore, by noting that
\[
E\left[ \int _0^T {\mathbb I}_{[-\varepsilon ,\varepsilon]} (X_t) dt \right] \leq \left( \inf _{x\in [-\varepsilon ,\varepsilon]}\frac{1}{\sigma (x)^2}\right) E\left[ \int _0^T {\mathbb I}_{[-\varepsilon ,\varepsilon]} (X_t) \sigma (X_t)^2 dt \right] .
\]
and applying the dominate convergence theorem again for $\varepsilon \downarrow 0$, we obtain \eqref{eq:proprefl05}.
\end{proof}

Theorem \ref{thm:refl} implies that in general the limit in law of strong Markov processes is not a strong Markov process.
The origin $0$ should be regarded as a reflecting boundary for the limit $X$ in Theorem \ref{thm:refl}, and a situation similar to Proposition \ref{prop:nonMarkov2} occurs.
As we have seen from Proposition \ref{prop:nonMarkov2} and the discussion after that, for the strong Markov property it is important for stochastic processes to be determined for all initial points in the state space.
Additionally to this fact, from Theorem \ref{thm:refl} it is seen that whether the stochastic process determined for all initial points in the state space or not is blinded by taking approximations. 

\begin{rem}
In \cite{FlIsRu} ``virtual solutions'' of SDEs with distributional drifts are considered.
However, Assumption 5 in \cite{FlIsRu}, because the derivative $\delta _0'$ of the Dirac delta function on ${\mathbb R}$ does not belong to the Sobolev space $H_p^{-\beta}$ for $\beta < (p-1)/p$.
\end{rem}

\section{Non-unique global solutions}\label{sec:nuSDE}

In this section, we introduce some examples of SDEs that have the uniqueness locally in time and global extendability of solutions, but does not have uniqueness globally in time.

\begin{thm}\label{thm:nuSDE}
Consider the SDE
\begin{equation}\label{eq:nuSDE}
\left\{ \begin{array}{rl}
dX_t &= {\mathbb I}_{{\mathbb R} \setminus \{ 0\}}(X_t)dB_t\\
X_0&=\xi 
\end{array}\right.
\end{equation}
where $B=(B_t; t\in [0,\infty ))$ is a one-dimensional Brownian motion and $\xi$ is a random variable independent of $B$.
If $P(\xi =0)=0$ and $E[|\xi |^2]<\infty$, then the followings hold.
\begin{enumerate}
\item \label{thm:nuSDE1}\eqref{eq:nuSDE} has the unique local-in-time solution $X=(X_t; t\in [0, \tau _0(X)])$ up to the first hitting time $\tau_0(X)$ at $0$, which is positive almost surely.

\item \label{thm:nuSDE2} Any solution $X=(X_t; t\in [0,T])$ has the estimate
\[
E\left[ \sup _{t\in [0,T]}|X_t|^2 \right] < \infty
\]
 for any nonrandom $T\in [0,\infty )$.
In particular, any solution $X$ is extended to a global-in-time solution $X=(X_t; t\in [0,\infty ))$.

\item \label{thm:nuSDE3} For any nonrandom $T\in [0,\infty )$ the two different processes $X_t = \xi + B_t$ and $\widetilde{X}_t = \xi + B_{t\wedge \tau _0(\widetilde{X})}$ are solutions to \eqref{eq:nuSDE}.
In particular, \eqref{eq:nuSDE} does not have the uniqueness of solutions on the time interval $[0,T]$.
\end{enumerate}
\end{thm}

\begin{proof}
The first assertion \ref{thm:nuSDE1} is follows from the facts that the solution $X$ to \eqref{eq:nuSDE} behaves as $B_t$ up to the hitting time $\tau_0(X)$ and that $P(\xi =0)=0$ implies $\tau_0(X)>0$ almost surely.
To prove \ref{thm:nuSDE2}, let $X=(X_t; t\in [0,\infty ))$ be a solution to \eqref{eq:nuSDE}.
Since the Burkholder-Davis-Gundy inequality implies
\begin{align*}
E\left[ \sup _{t\in [0,T]}|X_t|^2 \right] & \leq 2E\left[ \left| \xi \right| ^2\right] + 2 E\left[ \sup _{t\in [0,T]}\left| \int _0^t {\mathbb I}_{{\mathbb R} \setminus \{ 0\}}(X_s)dB_s \right| ^2\right] \\
&\leq 2E\left[ \left| \xi \right| ^2\right] + C E\left[ \int _0^T {\mathbb I}_{{\mathbb R} \setminus \{ 0\}}(X_s)ds \right] \\
&\leq 2E\left[ \left| \xi \right| ^2\right] + CT .
\end{align*}
where $C$ is an absolute constant, \ref{thm:nuSDE2} follows.
Since it is easy to see that the two continuous stochastic processes
\[
X_t := \xi + B_t \quad \mbox{and} \quad \widetilde{X}_t := \xi + B_{t\wedge \tau _0(\widetilde{X})}
\]
are solutions to \eqref{eq:nuSDE}, \ref{thm:nuSDE3} follows.
\end{proof}

\begin{rem}
\begin{itemize}
\item The first assertion of Theorem \ref{thm:nuSDE}\ref{thm:nuSDE2} implies that any solution does not explode in finite time.
Hence, any solution $(X_t; t\in [0,\tau ))$ to \eqref{eq:nuSDE} with a stopping time $\tau$ can be extended to a continuous stochastic process $(\bar{X}_t; t\in [0,\tau ])$.
Therefore, we have a global-in-time solution $Z$ by
\[
Z_t = \left\{ \begin{array}{cl}
\bar{X}_t, & t\in [0,\tau ],\\
Y_{t-\tau}, &t\in (\tau ,\infty ),
\end{array}\right.
\]
where $(Y_t; t\in [0,\infty ))$ is another solution with initial condition $Y_0 = \bar{X}_\tau$ whose driving noise is $\widetilde{B}_t := B_{t+\tau} - B_\tau$.
Here, note that the existence of solutions for all initial points follows from \ref{thm:nuSDE3}, in particular such $Y$ exists.

\item The SDE \eqref{eq:nuSDE} is also treated in \cite{Ba}, and infinitely many solutions are constructed in \cite{Ba} in view of sticky points.
In particular, there exist infinitely many strong Markov processes such that they satisfy \eqref{eq:nuSDE} and the Lebesgue measures of their spending time at $0$ equals to $0$.
\end{itemize}
\end{rem}

The point of Theorem \ref{thm:nuSDE} is that the SDE \eqref{eq:nuSDE} does not have the uniqueness of the solution starting at $0$, and hence solutions of the SDE \eqref{eq:nuSDE} have branches at $0$.
The trick to make the local-in-time unique and globally extendable, but not unique globally in time, is that the initial low do not have a positive mass at the branching points almost surely, i.e. $P(X_0=0)=0$ in the case of \eqref{eq:nuSDE}.
We usually do not assume $P(X_0=0)=0$, and the assumption seems eccentric.
However, two different solutions to \eqref{eq:nuSDE} in Theorem \ref{thm:nuSDE} are obtained by the unique limits of the solutions to different approximations of the SDE, as follows.

\begin{thm}\label{thm:nuSDEapprox}
Let $\xi \in {\mathbb R}$.
\begin{enumerate}
\item \label{thm:nuSDEapprox1}
Let $\varepsilon >0$ and consider the SDE
\begin{equation}\label{eq:nuSDEapprox1}
\left\{ \begin{array}{rl}
dX_t^{\varepsilon} &\displaystyle = \left[ \left( \frac{1}{\varepsilon} \sqrt{|X_t^\varepsilon|}\right) \wedge 1 \right] dB_t\\
X_0^{\varepsilon} &=\xi
\end{array}\right.
\end{equation}
where $B$ is a one-dimensional Brownian motion.
Then, for each $\varepsilon >0$ the solution $X^\varepsilon = (X_t; t\in [0,\infty ))$ to \eqref{eq:nuSDEapprox1} exists pathwise-uniquely, and the limit $X$ of $X^\varepsilon$ in the topology $L^2(\Omega ; C([0,\infty )))$ as $\varepsilon \downarrow 0$ exists and is given by $X_t = \xi +B_{t\wedge \tau _0(X)}$ where $\tau _0(X)$ is the first hitting time of $X$ at $0$.
In particular, the limit $X$ is a solution to \eqref{eq:nuSDE}.

\item \label{thm:nuSDEapprox2}
Let $\varepsilon >0$ and consider the SDE
\begin{equation}\label{eq:nuSDEapprox2}
\left\{ \begin{array}{rl}
dY_t^{\varepsilon} &\displaystyle = {\mathbb I}_{{\mathbb R} \setminus \{ 0\}}(Y_t^{\varepsilon}) dB_t + \varepsilon d\widetilde{B}_t\\
Y_0^{\varepsilon} &= \xi
\end{array}\right.
\end{equation}
where $B$ and $\widetilde{B}$ are two independent one-dimensional Brownian motions.
Then, for each $\varepsilon >0$ the solution $Y^\varepsilon = (Y_t; t\in [0,\infty ))$ to \eqref{eq:nuSDEapprox2} exists pathwise-uniquely, and the limit $Y$ of $Y^\varepsilon$ in the topology $L^2(\Omega ; C([0,\infty )))$ as $\varepsilon \downarrow 0$ exists and is given by $Y_t = \xi +B_t$.
In particular, the limit $Y$ is a solution to \eqref{eq:nuSDE}.
\end{enumerate}
\end{thm}

\begin{proof}
First we prove \ref{thm:nuSDEapprox1}.
Since the function $x\mapsto \left( \frac{1}{\varepsilon} \sqrt{|x|}\right) \wedge 1$ is $(1/2)$-H\"older continuous, by the Yamada-Watanabe theorem (cf. Theorem 1 in \cite{YW} or Theorem 3.2 of Chapter IV in \cite{IW}) we have the existence and pathwise uniqueness of the solution $X^\varepsilon$ to \eqref{eq:nuSDEapprox1} for any $X_0^\varepsilon \in {\mathbb R}$.
In particular, $X^\varepsilon$ satisfies the strong Markov property (cf. Section 5 of Chapter IV in \cite{IW}).
These and the fact that $X_t^\varepsilon =0$ for $t\in [0,\infty )$ is the solution to \eqref{eq:nuSDEapprox1} with the initial condition $X_0^\varepsilon =0$, imply that
\[
X_t^\varepsilon = X^\varepsilon_{t\wedge \tau _0 (X^\varepsilon)}, \quad t\in [0,\infty )
\]
almost surely.
Let $X_t = \xi +B_{t\wedge \tau _0(X)}$.
Then, for any $T>0$, by the Burkholder-Davis-Gundy inequality we have
\begin{align*}
&E\left[ \sup _{t\in [0, \tau _0 (X^\varepsilon)\wedge \tau _0(X) \wedge T]} \left| X_t^\varepsilon -X_t \right| ^2 \right] \\
&= E\left[ \sup _{t\in [0,T]} \left| \int _0^{\tau _0 (X^\varepsilon)\wedge \tau _0(X) \wedge T} \left\{ \left[ \left( \frac{1}{\varepsilon} \sqrt{|X_s^\varepsilon|}\right) \wedge 1 \right] -1 \right\} dB_s \right| ^2 \right] \\
&\leq C E\left[\int _0^{\tau _0 (X^\varepsilon)\wedge \tau _0(X) \wedge T} \left\{ \left[ \left( \frac{1}{\varepsilon} \sqrt{|X_s^\varepsilon|}\right) \wedge 1 \right] -1 \right\} ^2 ds \right] .
\end{align*}
Hence, the dominate convergence theorem yields
\begin{equation}\label{eq:nuSDEapprox1-1}
\lim _{\varepsilon \downarrow 0} E\left[ \sup _{t\in [0, \tau _0 (X^\varepsilon)\wedge \tau _0(X) \wedge T]} \left| X_t^\varepsilon -X_t \right| ^2 \right] =0, \quad T>0.
\end{equation}
From \eqref{eq:nuSDEapprox1-1} it is easy to see that 
\[
\lim _{\varepsilon \downarrow 0} P \left( \tau _0 (X^\varepsilon) \wedge T \neq \tau _0(X) \wedge T \right) =0, \quad T>0.
\]
From this fact, \eqref{eq:nuSDEapprox1-1} and the fact that $X_t^\varepsilon = X_t =0$ for $t\geq \tau _0 (X^\varepsilon) \vee \tau _0(X)$ almost surely, by the dominate convergence theorem we obtain the convergence of $X^\varepsilon$ to $X$ in the topology $L^2(\Omega ; C([0,\infty )))$.
The part ``In particular'' has been obtained in Theorem \ref{thm:nuSDE}\ref{thm:nuSDE3}.

Next we prove \ref{thm:nuSDEapprox2}.
Let $Y^\varepsilon$ be a solution to \eqref{eq:nuSDEapprox2}.
Note that similarly to the proof of Theorem \ref{thm:nuSDE}, we have
\begin{equation}\label{eq:nuSDEapprox2-1}
E\left[ \sup _{t\in [0,T]}|Y_t^\varepsilon |^2 \right] \leq C(|\xi|^2+ T (1+\varepsilon ^2))
\end{equation}
for any $T>0$ and an absolute constant $C$.
For any $\delta >0$ take $g_\delta \in C_0({\mathbb R})$ such that ${\mathbb I}_{[-\delta , \delta]} \leq g_\delta \leq {\mathbb I}_{[-2\delta , 2\delta]}$ and let
\begin{equation}\label{eq:defG}
G_\delta (x) := \int _{-2\delta}^x \left( \int _{-2\delta}^y g_\delta (z) dz \right) dy , \quad x\in {\mathbb R}.
\end{equation}
Then, the It\^o formula implies
\begin{align*}
E[G_\delta (Y_t^\varepsilon )] - E[G_\delta (\xi )] &= \frac{1}{2} \int _0^t E\left[ g_\delta (Y_s^\varepsilon ) {\mathbb I}_{{\mathbb R} \setminus \{ 0\}}(Y_s^{\varepsilon}) \right] ds + \frac{\varepsilon ^2}{2} \int _0^t E\left[ g_\delta (Y_s^\varepsilon ) \right] ds \\
&\geq \frac{\varepsilon ^2}{2} \int _0^t E\left[ {\mathbb I}_{[-\delta , \delta]} (Y_s^\varepsilon ) \right] ds.
\end{align*}
Since $|G_\delta (x)| \leq 4\delta (|x|+2\delta )$, in view of \eqref{eq:nuSDEapprox2-1} we have
\[
\int _0^t E\left[ {\mathbb I}_{[-\delta , \delta]} (Y_s^\varepsilon ) \right] ds \leq C \delta \varepsilon ^{-2} (\delta + |\xi|^2+ T (1+\varepsilon ^2)).
\]
for $t\in [0,T]$, where $C$ is an absolute constant.
Hence, by letting $\delta \downarrow 0$ we have
\[
E\left[ \int _0^t {\mathbb I}_{\{ 0\}} (Y_s^\varepsilon ) ds \right] = 0
\]
for $t\in [0,\infty )$ and $\varepsilon >0$.
This equality and the It\^o isometry imply that
\[
\int _0^t {\mathbb I}_{{\mathbb R} \setminus \{ 0\}}(Y_s^{\varepsilon}) dB_s = B_t
\]
for any $t\in [0,\infty )$ almost surely. Thus, we obtain that $Y_t^\varepsilon$ satisfies
\begin{equation}\label{eq:nuSDEapprox2-2}
Y_t^{\varepsilon} = \xi + B_t + \varepsilon \widetilde{B}_t, \quad t\in [0,\infty )
\end{equation}
almost surely.
Now it is easy to see that $Y^\varepsilon = (Y_t; t\in [0,\infty ))$ given by \eqref{eq:nuSDEapprox2-2} is the pathwise-unique solution to \eqref{eq:nuSDEapprox2}, and $Y^\varepsilon$ converges to $\xi + B$ in the topology $L^2(\Omega ; C([0,\infty )))$.
The part ``In particular'' has been obtained in Theorem \ref{thm:nuSDE}\ref{thm:nuSDE3}.
\end{proof}

In Theorem \ref{thm:nuSDEapprox}, we constructed two different solutions to \eqref{eq:nuSDE} by limits of the strong Markov processes which are the unique solutions starting at any given initial point to approximate SDEs.
On the other hand, by considering path-dependent SDEs, even without approximations, we can construct a SDE which has the local-in-time pathwise-unique solutions starting at any initial point and all the solutions are extendable to globally-in-time solutions, but the solutions globally in time are not unique.

\begin{thm}\label{thm:nuSDEpdep}
Consider the path-dependent SDE
\begin{equation}\label{eq:nuSDEpdep1}
\left\{ \begin{array}{rl}
dX_t &= \left( {\mathbb I}_{{\mathbb R} \setminus \{ 0\}}(X_t) + {\mathbb I}_{[0,1)}\left( \int _0^t |X_s| ds\right) \right) dB_t\\
X_0&=\xi ,
\end{array}\right.
\end{equation}
where $B$ is a one-dimensional Brownian motion and $\xi$ is a random variable independent of $B$.
\begin{enumerate}
\item \label{thm:nuSDEpdep1}\eqref{eq:nuSDEpdep1} has the pathwise-unique local-in-time solution $X=(X_t; t\in [0, \tilde{\tau} _0(X)])$ up to the stopping time $\tilde{\tau}_0(X)$ defined by 
\[
\tilde{\tau} _0(X) := \inf \left\{ t> 0; \int _0^t |X_s| ds \geq 1 \ \mbox{and}\ X_t =0 \right\},
\]
which is positive almost surely.

\item \label{thm:nuSDEpdep2} When $E[|\xi|^2]<\infty$, any solution $X$ to \eqref{eq:nuSDEpdep1} has the estimate
\[
E\left[ \sup _{t\in [0,T]}|X_t|^2 \right] < \infty
\]
 for any nonrandom $T\in [0,\infty )$.
In particular, any solution $X$ is extended to a global-in-time solution $X=(X_t; t\in [0,\infty ))$.

\item \label{thm:nuSDEpdep3} For any nonrandom $T\in [0,\infty )$, \eqref{eq:nuSDEpdep1} does not have the uniqueness of solutions on the time interval $[0,T]$.
\end{enumerate}
\end{thm}

\begin{proof}
The proof of Theorem \ref{thm:nuSDEpdep} is similar to that of Theorem \ref{thm:nuSDEapprox}.
For the proof of \ref{thm:nuSDEpdep1} we prove only the case that $\xi$ is nonrandom, because the general case follows from this case by consider it under the regular conditional probability measure given by the $\sigma$-field generated by $\xi$.
Let
\[
\bar{\tau} _1(X) := \inf \left\{ t> 0; \int _0^t |X_s| ds \geq 1\right\},
\]
Then, for any solution $X$ to \eqref{eq:nuSDEpdep1} satisfies
\[
X_t = \xi + \int _0^t \left( {\mathbb I}_{{\mathbb R} \setminus \{ 0\}}(X_s) +1 \right) dB_s , \quad t\in [0, \bar{\tau} _1(X)] .
\]
For $\delta >0$, define $g_\delta \in C_0({\mathbb R})$ and $G_\delta \in C^2({\mathbb R})$ as in the proof of Theorem \ref{thm:nuSDEapprox} (see \eqref{eq:defG}).
By the It\^o formula we have
\begin{align*}
E[G_\delta (X_{t\wedge \bar{\tau} _1(X)})] - E[G_\delta (\xi )]
&= \frac{1}{2} E\left[ \int _0^{t\wedge \bar{\tau} _1(X)} g_\delta (X_s) \left( {\mathbb I}_{{\mathbb R} \setminus \{ 0\}}(X_s) + 1 \right) ^2ds \right] \\
&\geq \frac{1}{2} E\left[ \int _0^{t\wedge \bar{\tau} _1(X)} {\mathbb I}_{[-\delta , \delta]} (X_s) ds \right] .
\end{align*}
Similarly to the proof of Theorem \ref{thm:nuSDEapprox}\ref{thm:nuSDEapprox2}, we have
\[
E\left[ \int _0^{t\wedge \bar{\tau} _1(X)} {\mathbb I}_{[-\delta , \delta]} (X_s) ds \right] \leq C \delta \varepsilon ^{-2} (\delta + |\xi|^2 + t (1+\varepsilon ^2)).
\]
for $t\in [0,\infty )$, where $C$ is an absolute constant.
Hence, by letting $\delta \downarrow 0$ and $t\uparrow \infty$ we have
\[
E\left[ \int _0^{\bar{\tau} _1(X)} {\mathbb I}_{\{ 0\}} (X_s) ds \right] = 0.
\]
This equality and the It\^o isometry imply that
\begin{equation}\label{eq:nuSDEpdep01}
X_t = \xi + 2B_t , \quad t\in [0, \bar{\tau} _1(X)]
\end{equation}
almost surely.
This yields the uniqueness of the solution on the time interval $[0, \bar{\tau} _1(X)]$.
From \eqref{eq:nuSDEpdep01} we also have that $0< \bar{\tau} _1(X)< \infty$ almost surely.
It is easy to see that $(X_t; t\in [0, \bar{\tau} _1(X)])$ given by \eqref{eq:nuSDEpdep01} satisfies \eqref{eq:nuSDEpdep1} on $[0, \overline{\tau} _1(X)]$.
Let $Y_t := X_{t+\bar{\tau} _1(X)}$ ($t\in [0,\infty )$).
Then, $Y_t$ satisfies \eqref{eq:nuSDE} with replacement of $\xi$ by $X_{\overline{\tau} _1(X)}$.
Note that $X_{\overline{\tau} _1(X)}\neq 0$ almost surely, because of \eqref{eq:nuSDEpdep01}.
Hence, by Theorem \ref{thm:nuSDE} $Y_t$ is uniquely determined for $t\in [0, \tau _0(Y)]$ where
\[
\tau _0 (Y):= \inf \{ t>0; Y_t=0\} .
\]
Since $\tilde{\tau} _0(X) = \bar{\tau} _1(X) + \tau _0(Y)$, we obtain \ref{thm:nuSDEpdep1}.

Similarly to the proof of Theorem \ref{thm:nuSDE}\ref{thm:nuSDE2}, we have \ref{thm:nuSDEpdep2} from the boundedness of the coefficient of \eqref{eq:nuSDEpdep1}.

Let $T>0$.
Since \eqref{eq:nuSDEpdep01} implies that $P(\bar{\tau} _1(X)<T )>0$, we obtain \ref{thm:nuSDEpdep3} from Theorem \ref{thm:nuSDE}\ref{thm:nuSDE3} and the fact that any stochastic process $X$ which satisfies 
\[
X_t = \left\{ \begin{array}{cl}
\xi + 2B_t , & t\in [0, \bar{\tau} _1(X)]\\
Y_{t-\bar{\tau} _1(X)}, & t\in [\bar{\tau} _1(X),\infty )
\end{array} \right.
\]
where $Y_t$ satisfies \eqref{eq:nuSDE} with replacement of $\xi$ by $X_{\overline{\tau} _1(X)}$, is a solution to \eqref{eq:nuSDEpdep1}.
\end{proof}

The example appeared in Theorem \ref{thm:nuSDEpdep} is of a path-dependent SDE.
However, since the dependence of the path is very simple, we can transform it to a multidimensional SDE which is not path-dependent, as follows.

\begin{cor}\label{cor:nuSDE2dim}
Consider the two-dimensional SDE
\begin{equation}\label{eq:nuSDE2dim}
\left\{ \begin{array}{rl}
d\left( \begin{array}{c}X_t\\ Y_t \end{array}\right)  &= \left( \begin{array}{c} \left( {\mathbb I}_{{\mathbb R} \setminus \{ 0\}}(X_t) + {\mathbb I}_{[0,1)}\left( Y_t \right) \right) dB_t \\ |X_t|dt \end{array}\right) \\
\left( \begin{array}{c}X_0\\ Y_0 \end{array}\right) &= \left( \begin{array}{c}\xi \\0 \end{array}\right) ,
\end{array}\right.
\end{equation}
where $B$ is a one-dimensional Brownian motion and $\xi$ is a random variable independent of $B$.
\end{cor}

\begin{proof}
For any solution $X$ to \eqref{eq:nuSDEpdep1}, by letting
\[
Y_t := \int _0^t |X_s|ds,
\]
we have a solution $(X,Y)$ to \eqref{eq:nuSDE2dim}.
On the other hand, any solution $(X,Y)$ to \eqref{eq:nuSDE2dim} satisfies
\[
Y_t = \int _0^t |X_s|ds.
\]
Hence, $X$ is a solution to \eqref{eq:nuSDEpdep1}.
Clearly, the relation between the solutions to \eqref{eq:nuSDEpdep1} and  \eqref{eq:nuSDE2dim} is a one-to-one correspondence.
Therefore, Corollary \ref{cor:nuSDE2dim} follows from Theorem \ref{thm:nuSDEpdep}.
\end{proof}

\begin{rem}
In this section we consider a one-dimensional SDE, which has branches of solutions at one point in the state space.
We chose this SDE, because it is the most simple example without the uniqueness of solutions.
Indeed, we have given different solutions explicitly.
We remark that the same phenomena occur also for other SDEs which have branches of solutions at non-polar sets.
\end{rem}

\section{Shifted equations}\label{sec:shift}

In this section we consider relations between solutions to shifted equations and to original SDEs.
Shifted equations appear in approaches in singular SPDEs, in particular in the paracontrolled calculus (see \cite{GIP}).

First we consider general SDEs and their shifted equations. 
Consider a SDE
\begin{equation}\label{eq:SDEX}
\left\{ \begin{array}{rl}
dX_t &= \sigma _X (t,X_t)dB_t + b_X(t,X_t)dt \\
X_0&=x.
\end{array}\right.
\end{equation}
Let $k\in {\mathbb N}$ and consider a system of SDEs
\begin{equation}\label{eq:SDEY}
\left\{ \begin{array}{rl}
dY_t^i &= \sigma _{Y^i} (t,Y_t^1,Y_t^2, \dots , Y_t^k)dB_t + b_{Y^i} (t,Y_t^1,Y_t^2, \dots , Y_t^k)dt \\
Y_0^i&=y^i.
\end{array}\right.
\end{equation}
for $i=1,2,\dots ,k$.
Here, the dimensions of equations are not mentioned, because any dimension is allowed in next (trivial) proposition.
Let ${\rm Sol}_X$ and ${\rm Sol}_Y$ be the total sets of solutions $X$ and $(Y^1,Y^2,\dots ,Y^k)$ to \eqref{eq:SDEX} and \eqref{eq:SDEY}, respectively.
Note that any sense of the solutions is allowed for the definitions of ${\rm Sol}_X$ and ${\rm Sol}_Y$.

Trivially, the following proposition holds.
\begin{prop}\label{prop:shift1}
\begin{enumerate}
\item If ${\rm Sol}_Y \neq \emptyset$ and there exists a map $f_{Y\rightarrow X}: {\rm Sol}_Y \rightarrow {\rm Sol}_X$, then the existence of the solution to \eqref{eq:SDEY} implies the existence of the solution to \eqref{eq:SDEX} (i.e. ${\rm Sol}_X\neq \emptyset$).

\item If there exists a injective map $f_{X\rightarrow Y}: {\rm Sol}_X \rightarrow {\rm Sol}_Y$, then the uniqueness of the solution to \eqref{eq:SDEY} (i.e. $\# {\rm Sol}_Y =1$) implies the uniqueness of the solution to \eqref{eq:SDEX} (i.e. $\# {\rm Sol}_X =1$).
\end{enumerate}
\end{prop}

\begin{rem}
Since the solutions $X$ and $(Y^1,Y^2,\dots ,Y^k)$ to \eqref{eq:SDEX} and \eqref{eq:SDEY}, respectively, have no interaction with each other, the existence (resp. uniqueness) of the solution to \eqref{eq:SDEX} and of the solution to \eqref{eq:SDEY} is equivalent to the existence (resp. uniqueness) of the coupled equation of \eqref{eq:SDEX} and \eqref{eq:SDEY}.
\end{rem}

In view of the paracontrolled calculus, we consider a sufficient conditions that for a given solution $(Y^1,Y^2,\dots ,Y^k) \in C([0,\infty ); {\mathbb R}^d)^{\otimes k}$ to \eqref{eq:SDEY}, the stochastic process $X \in C([0,\infty ); {\mathbb R}^d)$ defined by
\begin{equation}\label{eq:shiftXY}
X=Y^1 + Y^2 + \cdots + Y^k
\end{equation}
satisfies \eqref{eq:SDEX}.
A sufficent condition is that
\begin{equation}\label{cond1} \left\{ \begin{array}{rl}
\displaystyle \sigma _X (\eta ^0 , \eta ^1+ \eta ^2 \dots + \eta ^k) &\displaystyle = \sum _{i=1}^k \sigma _{Y^i} (\eta ^0 , \eta ^1, \dots , \eta ^k), \quad (\eta ^0 , \eta ^1, \dots , \eta ^k) \in {\rm Supp}_Y, \\
\displaystyle b_X (\eta ^0 , \eta ^1+ \eta ^2 \dots + \eta ^k) &\displaystyle = \sum _{i=1}^k b_{Y^i} (\eta ^0 , \eta ^1, \dots , \eta ^k), \quad (\eta ^0 , \eta ^1, \dots , \eta ^k) \in {\rm Supp}_Y, \\
\displaystyle x&\displaystyle = \sum _{i=1}^k y^i,
\end{array}\right. \end{equation}
where ${\rm Supp}_Y$ is the support of the map $(t,\omega ) \mapsto (t,Y_t^1(\omega ),Y_t^2(\omega) , \dots , Y_t^k(\omega))$, i.e. ${\rm Supp}_Y:= \{ (t,Y_t^1(\omega ),Y_t^2(\omega) , \dots , Y_t^k(\omega)); t\in [0,T], \omega \in \Omega \}$.
Here, we note that \eqref{cond1} can be relaxed for $X$ to satisfy \eqref{eq:SDEX} by allowing exceptional sets of $(t,\omega )$ with respect to $dt\otimes P(d\omega )$.

Next, in view of Proposition \ref{prop:shift1}, we consider a sufficient condition for the existence of an injective map $f_{X\rightarrow Y}: {\rm Sol}_X \rightarrow {\rm Sol}_Y$ which satisfies \eqref{eq:shiftXY} with $(Y^1,Y^2,\dots ,Y^k) = f_{X\rightarrow Y} (X)$ for $X \in {\rm Sol}_X$.
In the case that $Y^1$ is given by $X$, and $Y^{i+1}$ is uniquely determined by $X, Y^1, Y^2, \dots , Y^i$ for $i=1,2,\dots , k-1$ inductively, it is easy to construct $f_{X\rightarrow Y}$.
For example, the following proposition holds.

\begin{prop}\label{prop:shift2}
Let $k\in {\mathbb N}$ such that $k\geq 2$, let $X$ be a solution to \eqref{eq:SDEX}, and let $y^i$ be given for $i=1,2,\dots ,k-1$.
Consider coupled SDEs with a solution $(Y^1,Y^2,\dots ,Y^{k-1})$:
\begin{equation}\label{eq:propshift2Y}
\left\{ \begin{array}{rl}
dY_t^i &= \tilde{\sigma} _{Y^i} (t,Y_t^1,Y_t^2, \dots , Y_t^i, X_t)dB_t + \tilde{b}_{Y^i} (t,Y_t^1,Y_t^2, \dots , Y_t^i, X_t)dt \\
Y_0^i&=y^i
\end{array}\right.
\end{equation}
for $i=1,2,\dots ,k-1$.
Assume that \eqref{eq:propshift2Y} is solved pathwise-uniquely and recursively for $i=1,2,\dots , k-1$.
Then, $Y_t^k := X_t - (Y_t^1+ Y_t^2 + \dots + Y_t^{k-1})$ satisfies
\begin{equation}\label{eq:propshift2Yk}
\left\{ \begin{array}{rl}
dY_t^k &= \tilde{\sigma} _{Y^k} (t,Y_t^1,Y_t^2, \dots , Y_t^k, X_t) dB_t + \tilde{b}_{Y^k} (t,Y_t^1,Y_t^2, \dots , Y_t^k, X_t )dt \\
Y_0^k&= y^k := x-(y^1+y^2+ \cdots y^{k-1})
\end{array}\right.
\end{equation}
where
\begin{align*}
\tilde{\sigma} _{Y^k} (t,\eta ^1,\eta ^2, \dots , \eta ^k, \xi ) &:= \sigma _X (t,\xi ) - \sum _{i=1}^{k-1} \tilde{\sigma} _{Y^i} (t,\eta ^1,\eta ^2, \dots , \eta ^i, \xi ), \\
\tilde{b} _{Y^k} (t,\eta ^1,\eta ^2, \dots , \eta ^k, \xi ) &:= b_X (t,\xi ) - \sum _{i=1}^{k-1} \tilde{b} _{Y^i} (t,\eta ^1,\eta ^2, \dots , \eta ^i, \xi )
\end{align*}
for $(t,\eta ^1,\eta ^2, \dots , \eta ^k, \xi) \in [0,T] \times ({\mathbb R}^d)^{\otimes (k+1)}$.
\end{prop}

\begin{proof}
This is trivial, because of the definitions of $\tilde{\sigma} _{Y^k}$ and $\tilde{b} _{Y^k}$.
\end{proof}

\begin{rem}
When we consider the implication of the (time-inhomogeneous) strong Markov property of $X$ form that of $(Y^1,Y^2,\dots ,Y^k)$, we need the uniqueness in law of $(Y^1,Y^2,\dots ,Y^k)$ for all $(s,y^1,y^2,\dots , y^k)\in {\rm Supp}_Y$ (See Section \ref{sec:PreSDE} or \cite[Theorem 5.1, Chapter IV]{IW}).
Here, we emphasize that for the condition ``for all $(s,y^1,y^2,\dots , y^k)\in {\rm Supp}_Y$'' is necessary, because the solution to \eqref{eq:SDEY} has to be determined uniquely for initial points $y_i =Y^i_\tau$ for $i=1,2,\dots ,k$ and any stopping time $\tau$.
\end{rem}

The discussion above is almost trivial, but is the background of the method of singular SPDEs.
Indeed, we can regard the strategy of the method, as follows.
\begin{itemize}
\item[(Step 1)] We regularize the target singular SPDE. Once we regularize the target equation, we have a solution $X^{(n)}$ to the regularized equation.

\item[(Step 2)] Applying an analogue of Proposition \ref{prop:shift2} to the SPDE and we construct the solution $(Y^{(n),1},Y^{(n),2},\dots ,Y^{(n),k})$ to shifted equations. Note that $X^{(n)} = Y^{(n),1} + Y^{(n),2} + \cdots + Y^{(n),k}$.

\item[(Step 3)] Showing the existence of the limit of $(Y^{(n),1},Y^{(n),2},\dots ,Y^{(n),k})$ as $n\rightarrow \infty$ (removing the regularization), we have the limit $X$ of $X^{(n)}$ as $n\rightarrow \infty$.

\item[(Step 4)] We define the solution to the target singular SPDE by $X$.
\end{itemize}
See Section \ref{sec:SSPDE} for the detail of the method in singular SPDEs.
Here, it should be remark that for Proposition \ref{prop:shift2} the existence of the solution $X$ to the original equation is necessary.

Now we see the delicateness of the discussion of the solutions to original equations from the shifted equations by explicit examples.

\begin{prop}\label{prop:shift3}
Consider SDEs for $X$ on ${\mathbb R}$ and for $(Y^1,Y^2)$ on ${\mathbb R}^2$, as follows.
\begin{align}
\label{eq:shiftX}&\left\{ \begin{array}{rl}
dX_t &=  {\mathbb I}_{{\mathbb R}\setminus \{ 0\}} (X_t) dB_t \\
X_0 & = x,
\end{array} \right.
\\
\label{eq:shiftY}&\left\{ \begin{array}{rl}
dY^1_t &= {\mathbb I}_{|y^1|>1} {\mathbb I}_{[0,\tau_0(Y^1+Y^2) )} (t) dB_t\\
dY^2_t &= {\mathbb I}_{|y^1|\leq 1} dB_t\\
(Y^1_0, Y^2_0)& = (y^1, y^2).
\end{array} \right.
\end{align}
Then, we have the following.
\begin{enumerate}
\item \label{prop:shift3-1} For all $(y^1,y^2) \in {\mathbb R}^2$, the solution $(Y^1,Y^2)$ to \eqref{eq:shiftY} exists uniquely.

\item \label{prop:shift3-2} Let $x\in {\mathbb R}$ and $(y^1,y^2) \in {\mathbb R}^2$.
If $x=y^1+y^2$, then $X:=Y^1+Y^2$ is a solution to \eqref{eq:shiftX}.

\item \label{prop:shift3-3} Let $x\in {\mathbb R}$.
Let $(Y_1,Y_2)$ be the solution to \eqref{eq:shiftY} with initial condition $(y_1,y_2) =(0,x)$, and let $(\widetilde{Y}_1,\widetilde{Y}_2)$ be the solution to \eqref{eq:shiftY} with initial condition $(y_1,y_2) =(2,x-2)$.
Then, $Y_1+Y_2$ and $\widetilde{Y}_1 + \widetilde{Y}_2$ are different solutions to \eqref{eq:shiftX} with the same initial condition $X_0=x$.
\end{enumerate}
\end{prop}

\begin{proof}
We have \ref{prop:shift3-1} by solving the SDEs of $(Y^1, Y^2)$ as
\begin{equation}\label{eq:propshift3-01}
\left( \begin{array}{c} Y_t^1\\ Y_t^2\end{array}\right) = \left\{ \begin{array}{c}
\left( \begin{array}{c} y^1 + B_{t\wedge \tau_0(Y^1+Y^2)}\\ y^2\end{array}\right), \quad \mbox{if}\ |y^1|>1,\\[4mm]
\left( \begin{array}{c} y^1\\ y^2 + B_t\end{array}\right), \quad \mbox{if}\ |y^1|\leq 1.
\end{array} \right.
\end{equation}
Let $x:=y^1+y^2$ and $X:=Y^1+Y^2$.
From \eqref{eq:propshift3-01} it holds that
\begin{equation}\label{eq:propshift3-02}
Y_t^1 +Y_t^2 = \left\{ \begin{array}{c}
(y^1 + B_{t\wedge \tau_0(Y^1+Y^2)}) + y^2 = x + B_{t\wedge \tau_0(X)}, \quad \mbox{if}\ |y^1|>1,\\[4mm]
y^1 + (y^2 + B_t) = x + B_t , \quad \mbox{if}\ |y^1|\leq 1.
\end{array} \right.
\end{equation}
Hence, we have \ref{prop:shift3-2} by Theorem \ref{thm:nuSDE}.
From \eqref{eq:propshift3-02}, \ref{prop:shift3-3} follows.
\end{proof}

\begin{rem}
Proposition \ref{prop:shift3} is prepared to see the delicateness of the initial conditions of shifted equations.
From Theorem \ref{thm:nuSDE} we see that the uniqueness of solutions to \eqref{eq:shiftX} does not hold, and that two different solutions to \eqref{eq:shiftX} are chosen by $|y^1|>1$ or $|y^1|\leq 1$.
\end{rem}

Similarly to Theorem \ref{thm:nuSDEapprox}, the SDEs concerned in Proposition \ref{prop:shift3} can be written as a limits of uniquely-solved SDEs, but depending on the initial condition, as follows.

\begin{thm}\label{thm:shift3approx}
Let $\varepsilon >0$ and consider SDEs for $(Y^{1,\varepsilon},Y^{2,\varepsilon})$ on ${\mathbb R}^2$, as follows.
\begin{align}
\label{eq:shiftYapprox}&\left\{ \begin{array}{rl}
dY^{1,\varepsilon}_t &\displaystyle = {\mathbb I}_{(1,\infty )}(|y^1|) \left[ \left( \frac{1}{\varepsilon} \sqrt{|Y^{1,\varepsilon}_t + Y^{2,\varepsilon}_t|}\right) \wedge 1 \right] dB_t\\
dY^{2,\varepsilon}_t &\displaystyle = {\mathbb I}_{[0,1]}(|y^1|) {\mathbb I}_{{\mathbb R}\setminus \{ 0\}} (Y^{1,\varepsilon}_t + Y^{2,\varepsilon}_t ) dB_t + \varepsilon d\widetilde{B}_t\\
(Y^1_0, Y^2_0)& = (y^1, y^2).
\end{array} \right.
\end{align}
Then, we have the following.
\begin{enumerate}
\item \label{thm:shift3approx-1} For all $(y^1,y^2) \in {\mathbb R}^2$, the solution $(Y^{1,\varepsilon},Y^{2,\varepsilon})$ to \eqref{eq:shiftYapprox} exists uniquely, and converges in the topology of $L^2(\Omega ; C([0,\infty )))$ as $\varepsilon \downarrow 0$.

\item \label{thm:shift3approx-2} Let $x\in {\mathbb R}$, $(y^1,y^2) \in {\mathbb R}^2$ and let $(Y^1,Y^2)$ be the limit in the topology of $L^2(\Omega ; C([0,\infty )))$ as $\varepsilon \downarrow 0$ of the unique solutions $(Y^{1,\varepsilon},Y^{2,\varepsilon})$ to \eqref{eq:shiftYapprox} obtained in \ref{thm:shift3approx-1}.
If $x=y^1+y^2$, then $X:=Y^1+Y^2$ is a solution to \eqref{eq:shiftX} in Proposition \ref{prop:shift3}.

\item \label{thm:shift3approx-3} Let $x\in {\mathbb R}$.
For each $\varepsilon >0$, let $(Y^{1,\varepsilon},Y^{2,\varepsilon})$ be the unique solution to \eqref{eq:shiftYapprox} with initial condition $(y_1,y_2) =(0,x)$, and let $(\widetilde{Y}^{1,\varepsilon},\widetilde{Y}^{2,\varepsilon})$ be the unique solution to \eqref{eq:shiftYapprox} with initial condition $(y^1,y^2) =(2,x-2)$.
Denote the limits $(Y^1, Y^2)$ and $(\widetilde{Y}^1, \widetilde{Y}^2)$ in the topology of $L^2(\Omega ; C([0,\infty )))$ as $\varepsilon \downarrow 0$ of $(Y^{1,\varepsilon},Y^{2,\varepsilon})$ and $(\widetilde{Y}^{1,\varepsilon},\widetilde{Y}^{2,\varepsilon})$, respectively.
Then, $Y^1+Y^2$ and $\widetilde{Y}^1 + \widetilde{Y}^2$ are different solutions to \eqref{eq:shiftX} with the same initial condition $X_0=x$.
\end{enumerate}
\end{thm}

\begin{proof}
Note that
\[ \left\{ \begin{array}{l}
Y^{2,\varepsilon}_t =y^2, \quad \mbox{if}\ |y^1| >1, \\
Y^{1,\varepsilon}_t =y^1, \quad \mbox{if}\ |y^1|\leq 1.
\end{array} \right. \]
From Theorem \ref{thm:nuSDEapprox}, \ref{thm:shift3approx-1} follows.
Since
\begin{equation}\label{eq:thmshift3approx01}
X= (Y^1+y^2) {\mathbb I}_{(1,\infty )}(|y^1|) + (y^1+Y^2) {\mathbb I}_{[0,1]}(|y^1|) ,
\end{equation}
\ref{thm:shift3approx-2} also follows from Theorem \ref{thm:nuSDEapprox}.

Now we prove \ref{thm:shift3approx-3}.
In view of \eqref{eq:thmshift3approx01} and Theorems \ref{thm:nuSDEapprox}, by the SDEs which $Y^{1,\varepsilon} +Y^{2,\varepsilon}$ and $\widetilde{Y}^{1,\varepsilon} + \widetilde{Y}^{2,\varepsilon}$ satisfy, we have
\[
Y_t^1+Y_t^2 = x+B_t, \quad \widetilde{Y}_t^1 + \widetilde{Y}_t^2 = x + B_{t\wedge \tau _0(\widetilde{Y}^1+\widetilde{Y}^2)}.
\]
Now it is easy to see that $Y^1+Y^2$ and $\widetilde{Y}^1 + \widetilde{Y}^2$ are different solutions to \eqref{eq:shiftX} (see Theorem \ref{thm:nuSDE}\ref{thm:nuSDE3}).
\end{proof}

\begin{rem}
Note that the SDEs in Proposition \eqref{prop:shift3} does not satisfy \eqref{cond1}.
On the other hand, the limit of \eqref{eq:shiftYapprox} in Theorem \ref{thm:shift3approx} is a shifted equation of \eqref{eq:shiftX}, because
\begin{align*}
&{\mathbb I}_{(1,\infty )}(|y^1|) \left[ \left( \frac{1}{\varepsilon} \sqrt{|Y^{1,\varepsilon}_t + Y^{2,\varepsilon}_t|}\right) \wedge 1 \right] dB_t + {\mathbb I}_{[0,1]}(|y^1|) {\mathbb I}_{{\mathbb R}\setminus \{ 0\}} (Y^{1,\varepsilon}_t + Y^{2,\varepsilon}_t) dB_t + \varepsilon d\widetilde{B}_t \\
&\mathop{\longrightarrow}^{\varepsilon \downarrow 0} {\mathbb I}_{(1,\infty )}(|y^1|) {\mathbb I}_{{\mathbb R}\setminus \{ 0\}}(Y^{1}_t + Y^{2}_t) dB_t + {\mathbb I}_{[0,1]}(|y^1|) {\mathbb I}_{{\mathbb R}\setminus \{ 0\}} (Y^{1}_t + Y^{2}_t) dB_t \\
&= {\mathbb I}_{{\mathbb R}\setminus \{ 0\}} (Y^{1}_t + Y^{2}_t) dB_t.
\end{align*}
\end{rem}

\begin{rem}
The examples in Proposition \eqref{prop:shift3} and Theorem \ref{thm:shift3approx} are SDEs whose coefficients depending on the initial condition $y^1$.
But, the dependence of the initial condition can be removed easily.
For example, we can replace \eqref{eq:shiftY} by
\[
\left\{ \begin{array}{rl}
dY^1_t &= {\mathbb I}_{|Y^3_t|>1} {\mathbb I}_{[0,\tau_0(Y^1+Y^2) )} (t) dB_t\\
dY^2_t &= {\mathbb I}_{|Y^3_t|\leq 1} dB_t\\
dY^3_t &= 0\\
(Y^1_0, Y^2_0, Y^3_0)& = (y^1, y^2, y^3).
\end{array} \right.
\]
\end{rem}

As mentioned in Section \ref{sec:SSPDE}, the solution $X$ to a singular SPDE is defined by $X:=Y_1+Y_2+ \cdots + Y_k$ with a solution $(Y_1, Y_2, \dots ,Y_k)$ to the shifted equation.
Note that we often choose a specific initial condition for the shifted equation.
Indeed, in \cite[Section 4]{AlKu1} and \cite[Section 4]{AlKu2}, for the shifted equation of the $\Phi^4_3$-stochastic quantization equations we choose the initial conditions so that $X_0^{N,(2),<}=0$ and $X_0^{M,N,(2),<}=0$, respectively.
In \cite[Section 3.3]{HKK1} and \cite[Section 4.3]{HKK2}, to construct the solution to the $\exp (\Phi )_2$-stochastic quantization equation, we choose the initial condition $0$ for the shifted equation.
It should be recalled that in many cases, in order for the renormalization constants to be independent of the time parameter, the initial condition of the solution to the linearized equation, which appear as a component of shifted equations (see Section \ref{sec:SSPDE}), is chosen from suitable distributions except a null set with respect to the free field measure.
We often choose the initial conditions like these, because of technical difficulties.
In view of Proposition \ref{prop:shift3} and Theorem \ref{thm:shift3approx}, we keep in mind that the solutions of singular SPDEs are some specific ones for original singular SPDEs, even if the solutions to shifted equations are unique for each initial conditions. 

\bibliographystyle{plain}
\bibliography{RemStochSystem.bib}

\end{document}